\documentclass[12pt]{article}

\usepackage{palatino}
\usepackage{amsfonts}
\usepackage{enumerate}
\usepackage{amsmath}
\usepackage{amsthm}
\usepackage{amssymb,exscale}
\usepackage[all]{xy}

\theoremstyle{plain} 
    \newtheorem{theorem}{Theorem}
    \newtheorem{lemma}[theorem]{Lemma}
    \newtheorem{proposition}[theorem]{Proposition}
    \newtheorem{corollary}[theorem]{Corollary}
    
    \newtheorem{conjecture}[theorem]{Conjecture}
\theoremstyle{definition} 
    
    \newtheorem{fact}[theorem]{Fact}
    
    \newtheorem{remark}[theorem]{Remark}

\newcommand{\scond}{{(q)}}

\def \footnotemark{}

\def\bet{\beta}
\def\gam{\gamma}

\def\eps{\varepsilon}

\def\lam{\lambda}

\def\Hb{{\bf{H}}}

\def\S{\mathbb{S}}

\def\e{{\bf e}}

\def\summ{\sum\limits}

\def\prodd{\prod\limits}

\def\tends{\rightarrow}

\def\one{{\bf 1}}

\def\l{\left}
\def\r{\right}
\def\<{\langle}
\def\>{\rangle}

\newcommand{\E}{\mbox{\bf E}}
\newcommand{\ev}{\mbox{\bf E}}

\def\bar{\overline}
\def\P{{\bf P}}

\def\essup{\mbox{essup}}

\newcommand{\bareta}{\bar \eta}

\newcommand{\EE}[1]{\E\l[#1\r]}
\newcommand\Tr{{\mbox{Tr}}}
\newcommand\mnote[1]{} 
\newcommand\be{\begin{equation*}}

\newcommand\ee{\end{equation*}}

\newcommand\ben{\begin{equation}}
\newcommand\een{\end{equation}}
\newcommand\bes{\begin{eqnarray*}}
\newcommand\ees{\end{eqnarray*}}

\newcommand{\dist}{\mbox{\rm dist}}
\newcommand{\supp}{\mbox{\rm supp}}

\newcommand{\Var}{\operatorname{Var}}
\newcommand{\tr}{\operatorname{tr}}

\newcommand{\sm}{{\raise0.3ex\hbox{$\scriptstyle \setminus$}}}

\newcommand{\Cb}{\mathbf C}

\def\l{\left}
\def\r{\right}

\def\lam{\lambda}

\def\eps{\epsilon}
\def\tends{\rightarrow}

\renewcommand{\phi}{\varphi}

\newcommand{\HH}{{\mathcal H}}
\newcommand{\Hess}{\operatorname{Hess}}
\newcommand{\lambdamax}{\lambda_{max}}
\newcommand{\SAO}{\text{\rm SAO}_\beta}
 
\author{Manjunath Krishnapur, Brian Rider \\ and B\'alint Vir\'{a}g}
\title{Universality of the Stochastic Airy Operator}

\begin{document}

\bibliographystyle{abbrv}
\maketitle

\begin{abstract}
We introduce a new method for studying universality of random matrices. Let $T_n$ be the
Jacobi matrix associated to the Dyson beta ensemble with uniformly convex polynomial potential. We show
that after scaling, $T_n$ converges to the Stochastic Airy operator. In particular, the top edge of the
Dyson beta ensemble and the corresponding eigenvectors are universal. As a byproduct, our work leads to conjectured operator limits for the entire family of
soft edge distributions.
\end{abstract}

\pagebreak

\tableofcontents

\section{Introduction}

The goal of this paper is to introduce a novel approach to universality of random matrices. We consider Dyson's beta ensembles: these are $n$ random points on the real line with probability density
\begin{equation}\label{e:Dysonbeta}
\frac{1}{Z_{n,\beta}}  e^{ - \beta n \sum_{k=1}^n V(\lambda_k)} \, \prod_{j < k} |
\lambda_j - \lambda_k |^{\beta},
\end{equation}
where $V$ is polynomial, and we assume  $V''\ge c_u>0$.  We show that the distribution of the top points converges to a universal limit that does not depend on $V$. There have been two approaches to universality for such ensembles.  The first, classical method (see the book of Deift \cite{Deiftbook}) is based on asymptotics of the orthogonal polynomials, and is tied to the special values $\beta=1,2,4$. The more recent method, carried out in the bulk by Bourgade, Erd\H os,  and Yau \cite{BourgErdYau1, BourgErdYau2}, is based on the study the dynamics given by versions of Dyson's Brownian motion.

Just as in the classical method, our starting point is the theory of orthogonal polynomials. Recall that for a probability measure $\pi$ supported on exactly $n$ points there exists a unique $n\times n$ Jacobi matrix  $T$ (i.e., tridiagonal symmetric matrix with positive off-diagonals), so that the spectral measure of $T$ at the first coordinate vector $e_1$ is $\pi$. Consider the Jacobi matrix $T_n=T_n(V,\beta)$ associated to the random measure with support points picked form \eqref{e:Dysonbeta} and with independent Dirichlet$(\beta/2,..,\beta/2)$ distributed weights \eqref{e:Dirichlet}.

We study the structure of the matrix $T_n$ and show that after scaling, it converges as an operator to a unique random limit depending on $\beta$ only. This, in particular, implies the universality of the joint distribution of top eigenvalues: its limit does not depend on $V$.
\begin{theorem}\label{t:main}
There exists a coupling of the random matrices $T_n$ on the same probability space and constants $\gamma, \vartheta, \mathcal E$ depending on $V$ only (and specified in Remark \ref{r:constants}) so that
a.s. we have
$$\gamma n^{2/3}(\mathcal E-T_n)\to \SAO$$
in the norm-resolvent sense: for every $k$ the bottom $k$th eigenvalue converges the and corresponding eigenvector converges in norm.
Here $\mathcal E-T_n$ acts on $\mathbb R^n\subset L^2(\mathbb R_+)$ with coordinate vectors $e_j=(\vartheta n)^{1/6}\one_{[j-1,j](\vartheta n)^{-1/3}}$. \end{theorem}
For the special values of $\beta=1,2,4$, Theorem 1 strengthens (to operator convergence) some previously known results \cite{DG}. For those cases, \eqref{e:Dysonbeta} describes the eigenvalue distribution of a random matrix $\Upsilon$ with real, complex or quaternion entries, respectively. Then, $T_n$ is simply $\Upsilon$ written in the basis given by the Gram-Schmidt procedure with input $e_1, \Upsilon e_1, \ldots, \Upsilon^{n-1} e_1$.

The limiting object $\SAO$ of Theorem \ref{t:main} is the Stochastic Airy Operator, a second order differential operator with random potential defined by
\begin{equation}
\label{eq:SAO}
   \SAO = - \frac{d^2}{dx^2} + x + \frac{2}{\sqrt{\beta}} W'(x).
\end{equation}
Here  $x \mapsto W(x)$ is a standard Brownian motion, and $\SAO$ acts on a dense subset of $L^2(\mathbb R^+)$ with Dirichlet boundary conditions.
$\SAO$ was introduced in Rider, Ramirez and Vir\'ag \cite{RRV}, where Theorem 1 was proved for $V=x^2/4$, when $T_n$ has a particularly simple form found by Dumitriu and Edelman \cite{DE}. The GOE and GUE are special cases. The paper \cite{RRV} also establishes some basic criteria for the convergence of random tridiagonal operators. The proof of Theorem \ref{t:main} relies on \cite{RRV}, but apart from that and a few classical facts about orthogonal polynomials it is self-contained.

The eigenvalue/eigenvector pairs $(\Lambda_k, f_k)$ of $\SAO$ can also be defined via the variational formalism. We iteratively define
\begin{equation}
\label{variational}
  \Lambda_k = \inf_{f \perp f_0, \dots f_{k-1}  \atop f \in L } \int_0^{\infty} [ (f')^2(x) + x f^2(x) ] \, dx +  {\frac{2}{\sqrt{\beta}}} \int_0^{\infty} f^2(x) dW_x,
\end{equation}
where $L$ is the space of functions satisfying $f(0)=0$, $\int_0^{\infty} f^2 =1$, along with $    \int_0^{\infty} [ (f')^2 + x f^2] < \infty$.
That working on the space $L$ makes the stochastic integral in \eqref{variational} sensible as well as the form bounded below (almost surely) is one part of what is proved in \cite{RRV}.

The top eigenvalue of -$\SAO$ has the so-called Tracy-Widom-$\beta$ distribution. The representation \eqref{eq:SAO} has been been used to study rank-one deformations, to give a quick derivation of the Painlev\'e formulas for the TW$_\beta$ distribution for $\beta=2,4$, \cite{BV1} and for precise tail bounds \cite{DV}.

Most of this paper consists of the proof of Theorem 1, and we will conclude the introduction with an outline and motivation. But first, a conjecture, which is supported by further evidence in Section \ref{s:nonregular}.

The empirical distribution of eigenvalues of $T_n$, without scaling, converge to the classical equilibrium measure form potential theory corresponding to $V$ (see Remark \ref{r:momentcond}
below). The convexity and analyticity of $V$ forces this measure to have a density which is decays like $x^{1/2}$ at the edges. As one might guess, this $x^{1/2}$ is crucial for the $\SAO$ limit. When $V$ is analytic, the possible decay rates are $x^{2k+1/2}$ for some integer $k$. We conjecture (see Conjecture \ref{c:sk}) that after scaling, $T_n$ in this case converges to the random operator
$$
   \mathcal{S}_{\beta, k} =  - \frac{d^2}{dx^2} + x^{\frac{1}{2k+1}} +  \frac{2}{\sqrt{\beta}} x^{- \frac{k}{2k+1}} W'(x).
$$
For $\beta=2$ the eigenvalue limits have been studied in \cite{Claeys} via the Riemann-Hilbert approach.

\subsection{Methods of the proof}

To explain, begin with the $\beta$-Hermite ensembles, i.e. the case $V=x^2/4$.
For this case, Theorem 1 had been conjectured to hold by Edelman and Sutton \cite{ES, Sutton}.  Their reasoning, as well as  the rigorous proof in \cite{RRV}, makes essential use of
the discovery by Dumitriu and Edelman of a simple tridiagonal matrix model for the $\beta$-Hermite ensembles \cite{DE}  (see also the earlier work of Trotter \cite{Trotter} for the classical $\beta$ case).
Let
\begin{equation}
\label{thematrix}
H_n  =  \frac{1}{\sqrt{n \beta}} \left[ \begin{array}{ccccc}  g_1 & \chi_{(n-1)\beta} &&& \\
\chi_{(n-1)\beta} &  g_2 & \chi_{(n-2)\beta} & &\\
&\ddots & \ddots & \ddots & \\
& & \chi_{2\beta } &  g_{n-1} & \chi_{\beta} \\
& & & \chi_{\beta} &  g_{n}  \\
 \end{array}\right],
\end{equation}
in which the $g_k$ are Gaussian random variable of
mean $0$ and variance 2, the $\chi_k$ are $\chi$ random variables indexed by the shape parameter, and all variable are independent save for the condition that $H_n$ is symmetric.  The fact is that the eigenvalues of $H_n$ realize the law \eqref{e:Dysonbeta} for $V=x^2/4$ and all $\beta >0$.

The heuristic behind the Edelman-Sutton conjecture can then be gleaned from the asymptotic assessment: in distribution, $\chi_{n-k} \sim \sqrt{n} + \frac{k}{n} + g$ for a standard Gaussian random variable $g$.  One thus sees that, to leading order, the top corner of $H_n - 2I_n $ resembles the discrete second derivative operator. The corrections can be viewed as an additive potential which is of type linear  plus Gaussian noise.  This heuristic guides the proof of \cite{RRV} in which the Stochastic Airy Operator is identified by showing that, after ``centering" $H_n$ by the appropriate second derivative operator, the running sum of the process of entries (the integrated potential) converges to
$ \frac{1}{2}x^2 + \frac{2}{\sqrt{\beta}} W(x)$. In this way  Tracy-Widom limits are obtained as a consequence of a simple functional central limit theorem.

Here we continue this approach. First, we establish a tridiagonal representation for general $V$. In Proposition \ref{p:matrix model} we show that the diagonal $(A_1,\ldots, A_n)$ and off-diagonal $(B_1,\ldots, B_{n-1})$ entries of $T_n$ have joint density
\begin{equation}\label{density}
ce^{-n\beta H)}, \qquad \mbox{where }H=H(a,b)=\tr(V(T))-\summ_{k=1}^{n-1}(1-k/n-1/(n\bet))\log(b_k) .
\end{equation}
It follows from the path expansion of $\tr(T)$ that while the entries are not independent any more, they have a certain Markov field property. Indeed, variables with indices that are more that $\deg V/2$ apart are conditionally independent given the variables in between. Our goal is to prove the required central limit theorem and tightness conditions for these variables.

The Markov field property suggests that one could understand the distribution of the $(A,B)$ through studying some equilibrium measure of a Markov chain. One
issue is that the distribution is not homogeneous in $k$. However, one expects that mixing happens reasonably fast (in time $\log n$). In particular, some
local metastable equilibria will develop, and that $(A,B)$ will be close to these local equilibria. In Section \ref{s:local minimizers} we will study the location of these equilibria, and derive some properties of it.

Section \ref{s:outline} contains a far more extensive outline of the proof: we recall the criteria established in \cite{RRV} and outline how they will be applied. Essentially, we have to establish tightness and a functional CLT for the variables $A,B$.

The first step is to give rough bounds on the minimizers of $H$. This is achieved in Section \ref{sec:minimizerbounds}. Then, in Section \ref{s:minimizerboundary} we study the minimizers of versions of the Hamiltonian $H$ \ref{density}, and their dependence on boundary conditions. In the next section we show that they are close to the local minimizers studied in Section \ref{s:local minimizers}. In Section \ref{s:boundingthefield} we bound the random variables $(A,B)$, and in Section \ref{s:Gaussian} we establish a Gaussian approximation.

One wrinkle is that first $O(\log n)$ stretch of $(A, B)$ variables do not have a universal behavior. As a $(\deg V/2-1)$-Markov process,
$k \mapsto (A_k, B_k)$ can be expected to require $O(\log n)$ steps to achieve local equilibrium.
Thus, we first study the submatrices  $T_{[ c \log n, n]}$.  The effect of the truncation can be controlled by a rank-one perturbation. This is studied in Section \ref{s:firststretch}. In Section \ref{s:meanandvariance} we compute the parameters of the CLT, and in Section \ref{s:mainresult} we complete the proof. Conjectures about the general (nonregular) edge case are discussed in Section \ref{s:nonregular}.

\section{Tridiagonal models}\label{s:tridiagonal}

Let $T=T(a,b)$ denote the symmetric tridiagonal matrix with $T_{i,i}=a_i$ for $i\le n$ and $T_{i,i+1}=T_{i+1,i}=b_i$ for $i\le n-1$.

Recall that for a symmetric matrix $M$ the  spectral measure of $M$  at a unit vector $v$ is the probability measure whose $k$th moment is $\langle v,M^kv \rangle$.
In the below we take $v$ and to be the first coordinate vector $\mathbf e_1$.

\begin{proposition}[Matrix model] \label{p:matrix model}
Let $(A,B)$ be sampled from the density
\ben \label{eq:matrixpdf}
 \exp \l\{-n\beta \l[\tr(V(T))-\summ_{k=1}^{n-1}(1-k/n-1/(n\bet))\log(b_k) \r]\r\}.
\een
Then the eigenvalues of $T_n=T(A,B)$ have joint density proportional to
\be
 \exp\{-n\bet \summ_{j=1}^n V(\lam_j) \}\prodd_{j<k}|\lam_j-\lam_k|^{\bet}.
\ee
Moreover, the weights $q^2_i$ of the spectral measure $\mu:=\sum_{j=1}^n q_j^2\delta_{\lambda_j}$ of  $T$ are independent with Dirichlet$(\frac{\beta}{2},\ldots ,\frac{\beta}{2})$ distribution.
\end{proposition}
Henceforth $T$ will mean this random tridiagonal matrix and we shall assume that $V$ is a polynomial with even degree and positive leading coefficient.
It is not hard to show that in the classical $\beta=1,2,4$  cases $T$ has the distribution of the random matrix $M$ chosen from the probability distribution
$$
\mathcal Z^{-1}\exp(- n \beta \Tr V(M))\,dM,
$$
written in the orthonormal basis obtained from $\mathbf e_1, M \mathbf e_1, M^2\mathbf e_1,\ldots $ via the Gram-Schmidt procedure.

\begin{proof} By definition, $\<T^k\e_1,\e_1\>=\int x^k d\mu(x)$ for any $k\ge 0$. Consider this equation for each $0\le k\le 2n-1$ and write them as
\bes
 \prod_{j=1}^\ell b_j^2 + f_{\ell}(a_j,b_j;j<\ell) &=& \summ_{j=1}^n q_j^2 \lambda_j^{2\ell} \ \ \mbox{ if } k=2\ell  \\
 a_{\ell+1}\prod_{j=1}^\ell b_j^2 + g_{\ell}(a_j,b_{j'};j<\ell, j'\le \ell) &=& \summ_{j=1}^n q_j^2 \lambda_j^{2\ell+1} \ \ \mbox{ if } k=2\ell+1
 \ees
 Equate the Jacobian determinant with respect to $(a,b)$ of the left side with the Jacobian determinant of the right side with respect to $(\lam,q^2)$ to find that,
 $$
 2^{n-1}\prodd_{j=1}^{n-1}b_j^{4(n-j)-1}\, da_1\cdots  da_n \,db_1\cdots db_{n-1} = \det(M)\, dq_1^2\cdots dq_{n-1}^2\, d\lambda_1\ldots d\lambda_n
 $$
 in which
 $$
  M=\l[\begin{array}{cccccc}
    \lambda_1-\lambda_n & \ldots &\lambda_{n-1}-\lambda_n & q_1^2 &\ldots & q_n^2 \\
    \lambda_1^{2}-\lambda_n^{2} & \ldots &\lambda_{n-1}^{2}-\lambda_n^{2} & 2q_1^2\lambda_1 &\ldots & 2q_n^2\lambda_n \\
    \vdots & \vdots & \vdots & \vdots & \vdots & \vdots  \\
    \lambda_1^{2n-1}-\lambda_n^{2n-1} & \ldots &\lambda_{n-1}^{2n-1}-\lambda_n^{2n-1} & q_1^2(2n-1)\lambda_1^{2n-2} &\ldots & q_n^2(2n-1)\lambda_n^{2n-2}
   \end{array}\r].
 $$
 Note here that $\sum_{k=1}^n q_k^2 =1$.
 A special case of the confluent Vandermonde determinant identity \cite{Meray} yields
 $$
 \det(M)=\prod_{i=1}^n q_i^2 \prod_{i<j}(\lam_i-\lam_j)^4.
 $$
 Further, there is the identity
 \begin{equation}
 \label{spectralmapid}
  \prod_{k=1}^{n-1}b_k^{2(n-k)} = \prod_{i=1}^n q_i^2 \ \prod_{i<j}(\lambda_i-\lambda_j)^2,
 \end{equation}
 see for example Section 3.1 of \cite{Deiftbook}, expressing the
 $b_k$ in terms of the spectral measure of $T$.

 It follows that, the measure with density $\exp \l\{-n\beta \tr(V(T))\r\}\prod_{k=1}^{n-1}b_k^{\beta(n-k)-1}$ on $(a,b)$ transforms to the measure
 \begin{equation}\label{e:Dirichlet}
 \l( e^{-n\beta\sum_{k=1}^n V(\lambda_k)} \prod_{i<j}|\lambda_i-\lambda_j|^{\bet}\,d\lambda_1 \cdots d\lambda_k \r)\l(\prod_{k=1}^nq_k^{\beta-2} dq^2_1 \cdots dq_{n-1}^2 \r)\end{equation}
 for $(\lambda,q^2)$. In particular,  $\lambda$ and $q$ are independent, $q^2$ has Dirichlet distribution with parameters $(\frac{\beta}{2},\ldots ,\frac{\beta}{2})$, and $\lambda$ has the desired Coulomb gas distribution.
\end{proof}

\begin{lemma}[Uniform Convexity]
\label{l:UniformConvexity} The function $(a,b) \mapsto
\tr V(a,b)$ is convex or uniformly convex along with
$V$. That is, its Hessian is bounded below by $c_u I$ for $c_u =\min_x V''(x)$.
\end{lemma}

\begin{proof}
That $\tr(V(a,b))$ is convex follows from Chandler
Davis's theorem \cite{CD}: any convex function of the eigenvalues is
also a convex function of the entries. For uniform convexity,
write $V=\tilde V(x)+c_u x^2/2$, where $c_u=\min_x V''(x)$.
\end{proof}

With now $V(a,b) := \tr V(a,b)$ it follows that
\begin{equation}
\label{Hdef}
H=H(a,b)=\tr(V(T))-\summ_{k=1}^{n-1}(1-k/n-1/(n\bet))\log(b_k),
\end{equation}
shares its convexity properties with $V$  as long as $\beta \ge 1$.  For $\beta < 1$
the same is true for $H$ restricted to the coordinates with indices $k \le n - 1/\beta$ with
the remaining $(a,b)$-values viewed as fixed.

\section{Local minimizers}
\label{s:local minimizers}

Intuitively, to first order the variables $(A,B)$
should be close to the global minimizers $(a^\circ,b^\circ)$ of the Hamiltonian $H$,  recall \eqref{Hdef}. Recall also that for
 $\beta < 1$, this minimizer is understood to be subject to the ``boundary" condition $b_k = 0$ for $k > n - \frac{1}{\beta}$.
 The actual minimizers are difficult to characterize, but, as will show, for indices away from the boundaries they are
 locally close to constant functions. The goal of this section is to describe these constants.

Introduce the local Hamiltonian: for fixed $x$,
\begin{equation}
\label{localH}
H^{(x)}=H^{(x)}(a,b)=\tr(V(C))-\sum_{k=1}^{n}(1-x)\log(b_k)
\end{equation}
where $C$ is the symmetric circulant matrix with main diagonal given by the $a$'s, first off-diagonal given by the $b$'s and zeros elsewhere.
This definition aims to mimic the local behavior of $H$ around the index about $k=xn$.

Since $H^{(x)}$ is convex, it has a unique minimizer; since a rotation of the indices does not change $H^{(x)}$, it follows that for the minimizer, all of the $a_i$'s have to be equal, and also all of the $b_i$'s have to be equal.  Assuming this, note that as long as $n>\deg(V)$, the expression $H^{(x)}(a,b)/n$ does not depend on $n$. Thus the location of the minimum
 \begin{equation}\label{bdagger1}
 (a^\dagger=a^\dagger(x),b^\dagger=b^\dagger(x))\in \mathbb R^2
 \end{equation}
 does not depend on $n$. As functions of $x \in [0,1]$ these define the local minimizers.
We will also have reason to consider:
\begin{equation}\label{bdagger}
  a^\dagger_k=a^\dagger(k/n+1/(n\beta)), \qquad b^\dagger_k=b^\dagger(k/n+1/(n\beta)),
\end{equation}
 the local minimizers corresponding to index $k$.

More concretely, introduce the function
$$
W(a,b)=\frac{1}{\dim C}\tr V(C),
$$
assuming  $\dim C>\deg V$.  Then $a^{\dagger}(x), b^{\dagger}(x) $, minimize the
expression
\begin{equation}
\label{eq:localfunction}
   W(a,b)  - (1-x) \log b.
\end{equation}
The function $W$ may also be written as in
\begin{equation}
\label{eq:Laurent}
W(a,b)=[1]V(a+b(z+1/z))
\end{equation}
where $[1]$ denotes the coefficient of the constant term in the
Laurent series in $z$. One way to understand \eqref{eq:Laurent} is by counting random walk
paths. Another is to note that $C_{0,1}$ corresponds to the
sum of the left and right shift operators on the discrete circle.
In the Fourier basis it corresponds to by the multiplication
operator of $z+1/z$, and traces of multiplication operators are
given by the constant term.

\begin{remark}[Constants in the main theorem]\label{r:constants}  We can now specify the constants $\gamma, \vartheta, \mathcal E$ in 
Theorem \ref{t:main} for easy reference. With $a^\dagger(x), b^\dagger(x)$ defined quickly by (\ref{eq:localfunction}, \ref{eq:Laurent}), we have $\tau =  - (a^{\dagger})'(0)  - 2 (b^{\dagger})'(0) $,
$ \gamma = (b^{\dagger}(0))^{-1/3} \tau^{-2/3}$ 
$\vartheta= b^\dagger(0)/\tau$, and 
$\mathcal E= a^{\dagger}(0)  + 2 b^{\dagger}(0) $.
\end{remark}

For the rest of this section, we will drop the $\dagger$ from the local minimizers.
The most basic properties of the local minimizers as function of $x$ are captured in the following.

\begin{proposition}
\label{p:analytic}
The minimizers $x \mapsto a(x), b(x)$ of $W(a,b) - (1-x) \log b$ are unique and analytic as
functions of $x \in (-\infty, 1)$. Furthermore, they are continuous from the left at $x=1$ with $b(1) =0$.
\end{proposition}

\begin{proof}
Again, uniform convexity implies that $(a(x),b(x))$ are well-defined and unique.  This is true even
as local extrema; there cannot be local maxima and inflection points. Differentiating \eqref{eq:localfunction}
we get that,
at $(a(x),b(x))$:
\begin{equation}\label{eq:optimizationequation}
W_1(a(x),b(x))=0, \qquad W_2(a(x),b(x))-(1-x)/b(x)=0.
\end{equation}
Here the indices denote partial derivatives in the $a$ and $b$ variables.

Thinking of these as a pair of functions $f$ of $a,b,x$, the $2\times 2$
Jacobian matrix in $a,b$ is exactly the Hessian of $W- (1-x)\log b$. By uniform convexity (Lemma \ref{l:UniformConvexity})
this is bounded away from zero for $a\in \mathbb R, 1-x,b\ge \eps$ and
hence also for some complex neighborhood of these sets. The analytic implicit function theorem
now implies that $(a(x),b(x))$ is an analytic function of $x\in (-\infty, 1)$. Note that when $V(x)=x^2$,
we have $b(x)= \frac{1}{2} \sqrt{1-x}$, and analyticity breaks down at $x=1$.

Next,  differentiating the version
\begin{equation}\notag
W_1(a(x),b(x))=0, \qquad b(x)W_2(a(x),b(x))=1-x
\end{equation}
of \eqref{eq:optimizationequation} we get
\begin{equation}\label{eq:secondderivativesofW}
a'W_{11}+b'W_{12}=0, \qquad (a'W_{21}+b'W_{22})b+b'W_2 =-1.
\end{equation}
As $W_{11}>0$ by uniform convexity, it follows from the first equation of \eqref{eq:secondderivativesofW}
that $b'(x)=0$ implies $a'(x)=0$, but this would contradict
the second equation of \eqref{eq:secondderivativesofW}. From the optimization
problem \eqref{eq:localfunction} it is also clear that as $x\to-\infty$, we get $b(x)\to \infty$, so $b'(x)< 0$ for all $x<1$.

Last, since $W$ is an even function of $b$, we see that $b(1)=0$. Testing the minimizer against $a=a(1), b=1-x$, we get
$$
W(a(1),1-x)-(1-x)\log (1-x) \ge W(a(x),b(x))-(1-x)\log b(x).
$$
Uniform convexity of $W$  at its minimizer gives the lower bound
$$
 W(a(1),0)+c(a(1)-a(x))^2+cb(x)^2 -(1-x) \log b(x).
$$
Comparing the upper and lower bounds then shows that
$$
c(a(1)-a(x))^2+cb(x)^2 \le W(a(1),1-x)-W(a(1),0)+(1-x)\log b(x) -(1-x)\log (1-x).
$$
Since $b(x)$ is decreasing and $W$ is continuous, the right hand side tends to $0$ as $x\downarrow 1$.
\end{proof}

Continuing we note that  $W$ is not an arbitrary two-variable polynomial. For
example, it satisfies
\begin{lemma}\label{lem:wlemma}
$$
4b W_{11}=b W_{22}  +  W_2
$$
where the indices refer to partial derivatives in the $a$ or $b$ variables.
\end{lemma}

\begin{proof}
Using the formulation \eqref{eq:Laurent} this reduces to
$$
[1]4b V''(a+by) =[1] \left(by^2 V''(a+by) + yV'(a+by)
\right)
$$
with $y=z+1/z$. In order to show this, by shifting and scaling
$V$, we may assume that $b=1$, $a=0$. Then by linearity, it enough
to consider $V'(x)=x^k$. Then we get
$$
[1]4ky^{k-1} = [1]\left(y^2ky^{k-1}+yy^k\right).
$$
which since $[1]y^k= \binom{k}{k/2}$ reduces to the combinatorial
identity
\[
4k
\binom{k-1}{\frac{k-1}{2}}=(k+1)\binom{k+1}{\frac{k+1}{2}}.
\qedhere\]
\end{proof}

 We close this section with a formula for the inverse Hessian of
  $W(a,b)-(1-x)\log(b)$ at  $a(x),b(x)$. Note that by the definition of $W$, this
 the inverse of the Hessian the local Hamiltionian \eqref{localH}
    evaluated at its minimizer, and restricted to the invariant
    subspace with basis $\xi_a,\xi_b$. Here the $\xi$ are the vectors that
    are $1$ at all $a$ and $b$ variables, respectively, and zero otherwise.

\begin{proposition}
\label{p:localcovariance}
Denote by $\Sigma(x)$ the inverse of the Hessian of
    $W(a,b)-(1-x)\log(b)$ at its minimizer $a(x),b(x)$. It holds that
\begin{equation}
\label{e:sigma}
\Sigma(x)= -b(x) \l[\begin{array}{cc} 4b'(x) & a'(x) \\ a'(x) & b'(x)  \end{array} \r].
\end{equation}
Further, $\Sigma$ is strictly positive and bounded for $x \in [0, 1)$, and it is recorded for later use that
\begin{equation}
 \label{eq:edgemonotone}
   b'(x) < 0, \quad  a'(x) + 2 b'(x) < 0
\end{equation}
for $x \in [0,1)$.
\end{proposition}

\begin{proof}
Start with the equations
\begin{equation*}
a'W_{11}+b'W_{12}=0, \qquad (a'W_{21}+b'W_{22})b+b'W_2 =-1
\end{equation*}
derived in the course of proving Proposition \ref{p:analytic}.
Now use Lemma \ref{lem:wlemma} to cancel  $W_2$ and $W_{22}$ from
the second equation, and
$$W_{12}=-\frac{a'}{b'}W_{11}$$ to cancel $W_{12}$. This
gives
\[ W_{11} = \frac{-1}{bb'(4-\frac{a'^2}{b'^2})}\]
and by Lemma \ref{lem:wlemma} and the second equation of
\eqref{eq:optimizationequation} we have
$$W_{22}=4W_{11}-\frac{1-x}{b^2}.$$
Finally  we can compute
\begin{eqnarray*}
\Sigma =\l[\begin{array}{cc} W_{11} & W_{12} \\W_{21} & W_{22}+\frac{1-x}{b^2}  \end{array}\r]^{-1}&=&
 \frac{1}{W_{11}}\l[\begin{array}{cc} 1 & -\frac{a'}{b'}\notag \\ -\frac{a'}{b'} & 4  \end{array}\r]^{-1} \\&=& \frac{1}{W_{11}(4-\frac{a'^2}{b'^2})}\l[\begin{array}{cc} 4 & \frac{a'}{b'} \\ \frac{a'}{b'} & 1  \end{array}\r]=
-b\l[\begin{array}{cc} 4b' & a' \\ a' & b'  \end{array}\r],
\end{eqnarray*}
as claimed.

That $\Sigma$ is positive is another consequence of uniform convexity along with the continuity of $a(x), b(x)$.  That $b'(x) < 0$ was already noted in (the proof of) Proposition
\ref{p:analytic}, but both claims in \eqref{eq:edgemonotone} now follow from testing the quadratic form $\Sigma$ against simple vectors.
\end{proof}

\begin{remark}
\label{r:momentcond}
The equations for $(a(x), b(x))$ can be put in another form.   Using the integral formula for the Laurent coefficient, (\ref{eq:optimizationequation}) is equivalent to
\begin{equation}
\label{eq:momentconditions}
  \frac{i}{2\pi} \int_{L_x}^{R_x} \frac{s V_x(s)  \, ds }{\sqrt{(s-L_x) (R_x-s)}} =1,  \  \  \
   \int_{L_x}^{R_x} \frac{V_x(s)  \, ds }{\sqrt{(s-L_x) (R_x-s)}} =  0,
\end{equation}
where
$$
  V_x(s) = \frac{1}{1-x} V(s),  \   \   L_x =  a(x) - 2 b(x),  \   \  R_x = a(x) + 2 b(x).
$$
This identifies $(L_x, R_x)$ as the left and right endpoints of support for the equilibrium measure $\mu_V$
associated with the family of potentials  $V_x, -\infty < x < 1$.  That is, with
$$
  \mu_V = \mbox{argmin} \int_{-\infty}^{\infty} V_x(s) \mu(ds) + \int_{-\infty}^{\infty} \int_{-\infty}^{\infty} \log \frac{1}{|s-t|}  \mu(ds) \mu(dt),
$$
and the (realized) infimum taken over all probability measures $\mu$, it is the case that $\supp \, \mu_V = [L_x, R_x]$.  So, with obvious notation,
\begin{equation}
\label{eq:edgevalue}
  \mathcal{E}(x)  = R_x =  a(x) +2 b(x).
\end{equation}
 From this point of
view \eqref{eq:momentconditions} form the so-called moment conditions used in the determination of $\mu_V$, see for example \cite{KuilMcL}.
The results there provide another proof that  $a(x), b(x)$ are real analytic with ${b}'(x) < 0$, ${L}'_x > 0$, ${R}'_x = \mathcal{E}_x'  < 0$.
\end{remark}

\section{Outline of the main argument}
\label{s:outline}

The starting point of our argument is the paper \cite{RRV},  which provides a set of conditions for
a sequence of tridiagonal random matrices to converge (in the norm-resolvent sense)
to their natural continuum limit.  We begin by repeating  the set-up from that paper, along with a
needed extension from \cite{BV1}.

\subsection{Limits of random tridiagonal operators}

Start with a sequence of discrete-time
$\mathbb R^2$-valued random sequences
$(y_{n,1,k},y_{n,2,k})$ for $1\le k \le n$ with the
convention that $y_{n, i, 0}=0$. Let $m_n=o(n)$ be
a scaling parameter.
 For each
$n$, build the $n\times n$ symmetric tridiagonal matrix $H_n$ with
$$
(2   m_n^2+m_n(y_{n,1,k}-y_{n,1,k-1}),k\ge 1)
$$
 on the diagonal and
$$
(-  m_n^2+ m_n(y_{n,2,k}-y_{n,2,k-1}), k \ge 1)
$$
below and
above the diagonal.

Defining  $y_{n,i}(x)=y_{n,i,\lfloor x m_n\rfloor
}1_{xm_n\in[0,n]}$ and $\triangle_n$ the discrete Laplacian on the scale $m_n$ (so $m_n$ is one over the discretization length), $H_n$ should be viewed as
$-  \triangle_n$  plus integrated potential
$y_{n,1}(x)+ 2y_{n,2}(x)$.  What is desired is that $H_n \rightarrow H= - \frac{d^2}{dx^2} + y'(x)$ (in a sense to be made precise)
if it holds that $ y_{n,1}(x)+ 2y_{n,2}(x)$ converges to a process $y(x)$ (in a sense to be made precise).

Consider the following:

\bigskip

\noindent {\em Assumption 1 (Tightness/Convergence) } There
exists a continuous process $x \mapsto y(x)$ with $y(0)=0$ such that
\begin{eqnarray}\nonumber
\big(y_{n,i}(x);\;x\ge 0\big) &&i=1,2\quad \mbox { are tight in law, }\\
\big(y_{n,1}(x)+ 2  y_{n,2}(x);\; x\ge 0\big) &\Rightarrow
&\big(y(x);\, x\ge 0\big) \quad \mbox{ in
law,}\label{condition1}
\end{eqnarray}
with respect to the Skorokhod topology of paths, see
\cite{EK} for definitions.

\bigskip

\noindent {\em Assumption 2 (Growth/Oscillation bound) }
There is a decomposition
 \begin{equation}\label{e.assumptiongrowth}
y_{n,i,k} =m_n^{-1}\sum_{\ell=1}^{k} \eta_{n,i,\ell} \, +
\, w_{n,i,k},
 \end{equation} along with
 deterministic,
unbounded nondecreasing continuous functions
$\bareta(x)>0,\zeta(x) \ge 1$, and random constants
$\kappa_n(\omega)\ge 1$ defined on the same probability
space which satisfy the following. The $\kappa_n$ are tight
in distribution, and, almost surely,
\begin{eqnarray}
 \bareta(x)/\kappa_n\;-\kappa_n\le\; \eta_{n,1}(x)  + \eta_{n,2}(x) &\le& \;\kappa_n(1+\bareta(x)),
\label{bounds1}
 \\
\label{bounds2}  0 \le  \eta_{n,2}(x)&\le & m_n^2  \\
 |w_{n,1}(\xi)-w_{n,1}(x)|^2  +  |w_{n,2}(\xi)-w_{n,2}(x)|^2   & \le& \kappa_n(1+\bareta(x)/\zeta(x)).
\label{bounds3}
\end{eqnarray}
for all $n$ and $x,\xi\in[0,n/m_n]$ with $|x-\xi|\le 1$.

\bigskip

The growth/oscillation bounds in particular imply that the target limit operator $H$ is almost surely densely defined on $L^2$ (of the positive half-line). In particular, it is made sensible through its quadratic form $\langle f,g \rangle_H = \int f' g' + \int f g y' $ after a suitable integration by parts in the second term.
 Eigenvalues and eigenvectors of $H$, which has discrete spectrum with probability one,  are also defined through the quadratic form: $(\Lambda, f)$ is an eigenvalue/eigenvector pair
if $\langle f,\phi \rangle_H = \Lambda \int f \phi$  for all $\phi \in C_0^{\infty}$. As for convergence of $H_n$ to $H$, the needed result from \cite{RRV}, as extended in \cite{BV1} reads:

\begin{theorem}

(i) [Theorem 5.1 of \cite{RRV}] \label{weak}
Given Assumption 1 and 2 above let $(\lambda_{n, k}, v_{n,k})$, $k=1,2,\dots$, denote the ordered
eigenvalues/eigenvectors of the matrices $H_n$. Similarly let $(\Lambda_k, f_k)$  denote
the ordered eigenvalues/eigenvectors of the operator $H$, taken with Dirichlet boundary conditions at the origin.
Assume that $H_n$ acts on  $\mathbb R^n$ as a subspace of $L^2(\mathbb R)$ with coordinate vectors $e_j=\sqrt{m_n}\,\one_{[j-1,j]m_n^{-1}}$.
Then $H_n, H$  can be coupled on a probability space so that a.s. we have $\lambda_{n,k} \to \Lambda_k$ and
$v_{n,k} \to_{L^2} f_k$.

(ii) [Theorem 2.10 of \cite{BV1}] The result is unchanged for certain perturbations of $H_n$ at its first entry.  In particular
let $e_{11}$ be the matrix with $11$-entry equal to one and otherwise zero and consider the family of matrices $H_n +z_n e_{11}$
with $H_n$ as above and
\begin{equation}
\label{e:perturbation}
 \frac{z_n +m_n^2}{m_n} \rightarrow \infty
\end{equation}
in probability.  Then the eigenvalues/eigenvectors of $H_n$ still converge to those of $H$ (in the manner described) again with Dirichlet conditions at the origin.
\end{theorem}

In \cite{RRV} and \cite{BV1} this was was given as a distributional convergence statement, but by the standard Skorokhod embedding theorem it can be phrased this way (and the actual proof goes through Skorokhod embedding, giving the claimed result).
%
%

\begin{remark}Let us clarify what is meant here by norm-resolvent convergence. Let some operators $H_n$ be defined on a domain $A_n$ of $L^2$, and let $H$ be defined on a subspace $A$. If $H_n$ and
$H$ are closed and their spectrum is real, then for any non-real complex $z$ the resolvents $(z-H_n)^{-1}$ and $(z-H)^{-1}$ can be defined on {\it all of }$L^2$ and are bounded there by $1/\Im z$. This is very useful since the original operators may have disjoint domains and would be hard to compare. 

Norm-resolvent convergence means that these bounded operators converge in norm for some (equivalently, all) such $z$.

It is easy to check that if $H_n,H$ are closed and $H$ has discrete spectrum bounded below with no multiple eigenvalues, then norm-resolvent convergence is equivalent to the following: for every $k$ the $k$th lowest eigenvalue of $H_n$ converges to that of $H$ and the corresponding eigenvector of $H_n$ converges in norm to that of $H$. See \cite{Weidman} for more on convergence of unbounded operators.
\end{remark}

The extension in \cite{BV1} goes beyond {\em subcritical} perturbations of the form \eqref{e:perturbation}. In particular, there it is proved that for critical perturbations
reading
$z_n/m_n + m_n \rightarrow \omega \in (-\infty, \infty)$ changes the boundary condition for $H$: from Dirichlet, $f(0)= 0$, to Robin $f'(0) = \omega f(0)$.

 We should also note that have slightly changed the condition \eqref{bounds2} in Assumption 2 from how things were stated in \cite{RRV}. There it was assumed $\eta_{n,i} \ge 0$ while the proof only actually requires the non-negativity of $\eta_{n,2}$. Also, since it is convenient to define $y_{n,2}$ here to be twice that from \cite{RRV} we have required $\eta_{n, 2} \le m_n^2$ on the right hand side of \eqref{bounds2} (rather than the requirement $\eta_{n,2} \le 2 m_n^2$ of (5.5) of \cite{RRV}).
   We will have more to say about  this condition at the end of this section.

\subsection{Application to the Dyson $\beta$ ensembles}

 To apply Theorem \ref{weak} one anticipates that $\mathcal{E} I - T_n$ has
 has a Laplacian term after scaling.  That is, the main
diagonal should be $-2$ times the off-diagonal, to leading order. Recalling the definition of the local minimizers,
one would then expect $-2b^{\dagger}(0)=-\mathcal E+a^{\dagger}(0)$ $-$ the top of $T$ presumably satisfying $
\sim \mbox{trigdiag}( b^{\dagger}(0) , a^{\dagger}(0), b^{\dagger}(0))$
to first order.
This provides an intuitive understanding of the formula
for the edge $\mathcal E = a^{\dagger}(0) + 2 b^{\dagger}(0)$ derived in Section  \ref{s:local minimizers}, see \eqref{eq:edgevalue}.

Based on the $\beta$-Hermite case ($V(x) = x^2/4$), we should rescale as in $ \gamma n^{2/3}  ( \mathcal{E} I - T)$ for some $\gamma = \gamma(V, \beta)$. This sets
 $m_n = \sqrt{ \gamma b^{\dagger}(0)} \, n^{1/3}$, and the appropriate choice is to let
$$
  \gamma  = \frac{1}{ \tau^{2/3}  (b^{\dagger}(0))^{1/3} }, \mbox{ with } \tau = - (  (a^{\dagger})'(0)  + 2 (b^{\dagger})'(0)  ),
$$
 so that
 \begin{equation}
 \label{eq:mn}
    m_n = ( b^{\dagger}(0)  n/\tau)^{1/3}.
 \end{equation}
From \eqref{eq:edgemonotone} and \eqref{eq:edgevalue} we have that  $\tau$ is positive and equals the derivative of the edge at zero, $\tau = \mathcal{E}'(0)$.  For the $\beta$-Hermite case $\gamma=\tau = b^{\dagger}(0) = 1$.

 One wrinkle is that, as alluded to the introduction, the first $O(\log n)$ stretch of $(A, B)$ variables do not have a universal behavior. As a $(\deg V/2-1)$-Markov process,
 $k \mapsto (A_k, B_k)$ can be expected to require $O(\log n)$ steps to achieve local equilibrium.
 Thus, one can only hope to apply
 the Theorem \ref{weak} as such directly to the submatrices  $T_{[ c \log n, n]}$ for some  $c$ = $c(V, \beta)$.  Our strategy will be to show that these truncated matrices satisfy part (i)
 of the theorem, while the effect of the truncation can be controlled by a rank-one perturbation, to which part (ii) of the theorem applies.

We therefore set for a suitably large constant $c$ and $m_n$ as defined in \eqref{eq:mn}:
\begin{equation}
\label{summedpotentials}
y_{n,1}(x) = m_n \sum_{k=  \lfloor c \log n \rfloor}^{ \lfloor x m_n \rfloor } (a^{\dagger}(0) - A_k )/ b^{\dagger}(0),
\quad
y_{n,2}(x) =   m_n \sum_{k=  \lfloor c \log n \rfloor }^{\lfloor x m_n \rfloor} (b^{\dagger}(0) - B_k)/ b^{\dagger}(0) .
\end{equation}
This choice naturally prompts the further definitions:
\begin{equation}
\label{noise}
w_{n,1}(x)=   m_n  \sum_{k=  \lfloor c \log n \rfloor}^{ \lfloor x m_n \rfloor } (a^{\dagger}_k - A_k ) / b^{\dagger}(0),
\quad
w_{n,2}(x)=    m_n \sum_{k= \lfloor c \log n \rfloor }^{ \lfloor x m_n  \rfloor } ( b^{\dagger}_k  -B_k)/  b^{\dagger}(0),
\end{equation}
and
\begin{equation}
\label{driftderivatives}
\eta_{n,1}(x)=     m_n^{2}  ( a^{\dagger}(0) - a^{\dagger}_{\lfloor x m_n \rfloor} )/  b^{\dagger}(0), \quad
\eta_{n,2}(x)=   m_n^{2}   ( b^{\dagger}(0) - b^{\dagger}_{\lfloor x m_n \rfloor} )/ b^{\dagger}(0).
\end{equation}
In this way we recognize $w_{n,i}$'s  as noise/oscillation terms and the $\eta_{n,i}$'s as (derivatives of) the drifts.

It is important to note that the tightness/convergence conditions on $y_{n,i}$ from Assumption 1 only require looking at the variable $x$ on bounded sets, that is,
$x= O(1)$.  On the other hand, the conditions on the noise and drift components set out in Assumption 2 require $x$'s that track the entire index set of the matrix,
or up to $x$ of order  $n^{2/3}$.  Since our control of the variables at the bottom of the matrix is not so sharp, it is more convenient put everything in the growth terms.  More
succinctly, what we actually do is to let:  for an $\epsilon  = \epsilon(V, \beta) > 0$  chosen below,
\begin{align}
\label{e:endcondadjust}
   &  \mbox{Retain } \eqref{noise}, \eqref{driftderivatives} \mbox{ for } x \le n^{2/3} (1 -\epsilon),
   \mbox{ otherwise set: } \\
   & w_{n,i} = 0, \,  \eta_{n,1}(x) =  m_n^2 (a^{\dagger}(0) - A_{\lfloor x m_n \rfloor} ) / b^{\dagger}(0),  \
     \eta_{n,2}(x) =   m_n^2 (b^{\dagger}(0) - B_{\lfloor x m_n \rfloor} ) / b^{\dagger}(0). \nonumber
\end{align}
In terms of matrix indices this cutoff occurs at $k = n(1-\epsilon)$.

After several sections of preliminary estimates (to show that the minimizers of $H$ are indeed well-approximated by the local minimizers, around which the field concentrates
well), Section \ref{s:Gaussian} contains in Proposition \ref{p:variation} the basic fluctuation result that allows Assumptions 1 and 2 above to be verified for the truncated matrices
$H_n = \gamma n^{2/3} (\mathcal{E} I - T_{[c \log n, n]})$.  Section \ref{s:firststretch} establishes that the effect of this truncation on the spectrum can be bounded by a suitable perturbation: that the spectrum of $\gamma n^{2/3} (\mathcal{E} I - T_{n})$ is bounded in terms of that for an ensemble of the form $H_n + z_n e_{11}$ with $z_n$ satisfying
requirement \eqref{e:perturbation} of part (ii) of Theorem \ref{weak}.  Our main result is then proved in Section \ref{s:mainresult}

\begin{remark}  The rather funny condition \eqref{bounds2}  is automatically satisfied by the $\eta_{n,2}$
defined in \eqref{driftderivatives}: the results of Section \ref{s:local minimizers} show that $x \mapsto b^{\dagger}(x) $ is nonnegative and decreasing.
\end{remark}

\section{Bounding minimizers}
\label{sec:minimizerbounds}

A first step toward showing that the global minimizers $(a^\circ, b^\circ)$ are well approximated by the local minimizers
$(a^{\dagger}, b^{\dagger})$ defined in Section \ref{s:local minimizers} is to develop some preliminary upper and lower bounds
for $(a^{\circ}, b^{\circ})$. Actually, in the course of the proof it will be natural (and necessary) to consider minimizers of $H$ restricted to
some subsets  $J_a$, $J_b$ of coordinates, the values of $a,b$ on the complement of $J_a$ (respectively $J_b$)  being fixed as boundary conditions. A collection of
bounds are therefore established in the context of such ``conditional minimizers".

\subsection{Upper bounds}

We are interested in the minimizer of
Hamiltonians similar to $H$ over entries $a_j, j\in J_a$ and
$b_j, j\in J_b$. For example, for the overall minimizer we take
$J_a=\{1,\ldots, n\}$, $J_b=\{1,\ldots,n-1\}$. Now let $I$ be the
set of integers at most $\deg V/2$ away from $J_a\cup J_b$.

Let $T_I$ be the finite minor of the doubly infinite tridiagonal matrix $T(a,b)$ corresponding to the indices
$I$, and we set
$$
V_I(a,b):=\tr(V(T_I(a,b)).
$$
Fix $y_2,y_3\in(0,1]$, and define
\begin{equation}
\label{e:localH2}
\mathcal H(a,b)=V_I(a,b)-\summ_{j\in J_b }\alpha_j\log b_j, \qquad |\alpha_j| \le y2 \mbox{ for }  j\in J_b
\end{equation}
The variables $a_j,j\in I\setminus J_a$, $b_j, j\in I\setminus J_b$ are considered fixed, or boundary conditions, and we are interested in the minimizer of $\mathcal H$
over  $a_j\in \mathbb R$ for $j\in J_a$  and
\begin{equation}
  \label{e:bcond}
b_j\ge   \begin{cases}
y_3&\mbox{for } \alpha_j< 0 \\ 0 &\mbox{otherwise},
\end{cases}
\end{equation}
for $j\in J_b$.  We employ
the convention that $\log 0=-\infty$.

The goal of this section is to show Theorem \ref{thm:upper bound}, a bound on the minimizers. It will be used in two ways. First, more roughly, if the $\alpha_i$ and the boundary conditions are bounded by a constant, then so are the minimizers. Second, it will be used for indices close to $n$ in our original problem,  where the $|\alpha_i|$ are small. In this case,
if the boundary conditions are close to $a(1),\,b(1)=0$, then we show that so are the minimizers.

\begin{theorem}\label{thm:upper bound}
There exists a constant $c$ depending on the polynomial $V$ only so that if $$
|a_k-a^\dagger(1)|\le y_1 \mbox{ for all } k \in I\setminus  J_a, \qquad \mbox{ and } \qquad
b_k\le y_1 \mbox{ for all } k \in I\setminus  J_b, $$ then
for the minimizers $a, b$ of $\mathcal H$ satisfying \eqref{e:bcond} and all $k\in I$ we have
$$
|a_k-a^\dagger(1)|+b_k\le c\max(y_1,y_1^{(\deg V)^2/2}, y_2\log (e/y_3),(y_2\log (e/y_3))^{\deg V/2})^{1/2}.$$
\end{theorem}

The first step in the strategy is to show that a bound on $\mathcal H(a,b)$ gives a bound on the
$\ell^2$-norm of $(a,b)$. We will then show that $\mathcal H(a,b)$ is small for the minimizers $(a,b)$ by comparing it to an explicit example.

\begin{lemma}
\label{l:underV}
Let $V$ be a polynomial satisfying $V(x)\ge c_1x^2$ for
$c_1>0$.  Then for $b_j$ satisfying \eqref{e:bcond}, the Hamiltonian \eqref{e:localH2} has the bound
$$
\mathcal H(a,b) \ge \frac{c_1}{2}\sum_{i\in I}(a_i^2+b_i^2) - |J_b|y_2((\log c_1)^- +\log(1/ y_3)).
$$
\end{lemma}

\begin{proof}
By the lower bound on $V$, the left hand side minus
$\frac{c_1}2\tr(T_I^2)$ is bounded below by
$$
\frac{c_1}{2}\tr(T_I^2)-\sum_{i\in
J_b} \alpha_i \log
b_i\ge \sum_{i\in J_b} (-\alpha_i \log b_i +c_1 b_i^2),
$$
where the terms can be minimized individually.
When $\alpha_i<0$,
and so $b_i\ge y_3$ a lower bound is $y_2\log y_3$, and the same lower bound holds
for $\alpha_i=0$.
When $\alpha_i>0$ the minimal
value is
\[
\frac{\alpha_i}2 (1- \log {\alpha_i}+\log (2c_1))\ge -y_2(\log c_1)^-. \qedhere
\]

\end{proof}

By comparing minimizers to some specific substitution, we get a bound that depends on the boundary conditions.

\begin{corollary}\label{c:l2bound} Assume that $V(a)$ is minimized at $a=0$ and $V(0)=0$. There exists a constant $c$ depending on $V$ and
so that
the conditional minimizer $(a,
b)$ of $\mathcal H$ satisfies
\begin{equation}\sum_{k\in I}(a_k^2+b_k^2)\le
c\gamma_4|I|+c(1+\gamma_4^{1-\deg V})\left(\sum_{k\in I\setminus J_a} a_k^{\deg V}+\sum_{k\in I\setminus J_b} b_k^{\deg V}\right)
\end{equation}
as long as \begin{equation}\label{e:g4y2y3}
\gamma_4\ge y_2\log(e/ y_3).
\end{equation}
\end{corollary}

\begin{proof}
Let $q$ denote the total of the two sums above, and let $d=\deg V$. Consider the candidate vector $a',b'$ in which all $a'_i$ for  $i\in J_a$ are set
identically $0$ and all $b'_i$ for $i\in J_b$ terms are identically equal to $y_2$, and the rest are given by the boundary conditions.
For some $c$ depending on $V$ only, we have
$$V_I(a',b') \le c\sum_{i\in I} \left(|a'_i|+ {a'_i}^{d}+ b'_i+{b'_i}^{d}\right)
$$
since in the expansion, each $a'_i,b'_i$ is contained in a bounded number of monomials and the coefficients are all bounded. Now using the given substitution and the fact that $y_2\le 1$, we get the upper bound
$$V_I(a',b') \le 2cy_2|I| + c\sum_{i\in I\setminus J_a} \left(|a'_i|+ {a'_i}^{d}\right)+ c\sum_{i\in I\setminus J_b}\left( b'_i+{b'_i}^{d}\right).
$$
Note that for any $x,\eta>0$ we have $x+x^d\le \eta^{1-d} x^{d} + \eta$, whence for any $\eta>0$
$$V_I(a',b') \le c|I|(2y_2+\eta) + c (1+\eta^{1-d}) q.$$
This, together with the bound $|I|y_2\log(1/ y_3)$ on the logarithmic terms then yields
\begin{equation}\label{e:Hp}
\mathcal H (a',b') \le c|I|(y_2\log(e/y_3)+\eta) + c (1+\eta^{1-d}) q.
\end{equation}
We have assumed that $V$ is minimized at 0, and $V(0)=0$, so $V$ is bounded below by $c_u x^2/2$, where $c_u$ is the uniform convexity constant form Lemma \ref{l:UniformConvexity}.
By Lemma \ref{l:underV} applied to $T_I$ we have the lower bound for the minimizer $a,b$
\begin{eqnarray*}
\mathcal H(a, b)\ge \frac{c_1}{2}\sum_{k\in I}(a_k^2+b_k^2)-y_2((\log c_2)^- +\log(1/ y_3))|I|.
\end{eqnarray*}
Together with \eqref{e:Hp} and $\mathcal H (a',b')\ge \mathcal H (a,b)$ gives
$$
\sum_{k\in I}(a_k^2+b_k^2) \le c|I|(y_2\log(e/y_3)+\eta) + c (1+\eta^{1-d}) q.
$$
We set $\eta=y_4$ and the claim follows.
\end{proof}

For the proof of Theorem \ref{thm:upper bound} this bound will be iterated. The following lemma isolates what we need for the iteration.

\begin{lemma}\label{l:iteration}
Given $\gamma_1,\gamma_2>0$, $\alpha\ge 2$, and a nondecreasing positive sequence $s_n$, assume that for $n=1,\ldots, n^*-1$ we
have
\begin{equation}\label{e:iteration}
  s_n\le (s_{n+1}-s_n)^\alpha/\gamma_1 + \gamma_2n.
\end{equation}
For any positive $\gamma_3, x_0$ satisfying
$$\gamma_3\le (32\gamma_1)^{\frac{-1}{\alpha-1}}, \qquad x_0\ge \max( (\gamma_2/\gamma_3)^{1-1/\alpha}-1,0)$$
set
$
f(x)=\gamma_3(x+x_0)^\frac{\alpha}{\alpha-1}.
$

If $s_1\ge f(1)$ then $s_n\ge f(n)$ for all $n=1,\ldots, n^*$.
\end{lemma}

\begin{proof}
We show this by induction on $n$. Note that $f(x)$ has
increasing derivative, so we have $f(x+1)-f(x) \le
f'(x+1)$. With $()_+$ denoting the positive part, we write
$$
s_{n+1}=  s_n+(s_{n+1}-s_n)\ge s_n +
(\gamma_1s_n-\gamma_1\gamma_2n)_+^{\frac1\alpha}\ge
f(n)+(\gamma_1f(n)-\gamma_1\gamma_2n)_+^{\frac1\alpha}
$$
using the bound on the derivative, we get the lower bound $
f(n+1)-f(n+1)'+(\gamma_1f(n)-\gamma_1\gamma_2 n)^{\frac1\alpha} $. So it
suffices to show that $ f(n+1)'\le
(\gamma_1f(n)-\gamma_1\gamma_2n)^{\frac1\alpha} $ This reduces to the
inequality
$$
\left(\frac{\alpha\gamma_3}{\alpha-1}\right)^\alpha
(n+x_0+1)^\frac{\alpha}{\alpha-1} \le
\gamma_1\gamma_3\,(n+x_0)^\frac{\alpha}{\alpha-1}-\gamma_1\gamma_2n.
$$
We first choose $\gamma_3$ so that the dominant terms satisfy
$$
\left(\frac{\alpha\gamma_3}{\alpha-1}\right)^\alpha
(n+x_0+1)^\frac{\alpha}{\alpha-1} \le
\frac{\gamma_1\gamma_3}{2}\,(n+x_0)^\frac{\alpha}{\alpha-1},
$$
which holds for all $n\ge 1, x_0\ge 0$ as long as $\gamma_3\le (32\gamma_1)^{\frac{-1}{\alpha-1}}$.
It suffices to check
$$
\gamma_1\gamma_2n\le \frac{\gamma_1\gamma_3}{2}
(n+x_0)^\frac{\alpha}{\alpha-1},
$$
which follows from our assumptions.
\end{proof}

\begin{proof}[Proof of Theorem \ref{thm:upper bound}]
By shifting the argument and the value of $V$ by a constant we may assume that its minimizer $a^\dagger(1)$ is zero and $V(0)=0$.
By deleting unnecessary indices, we may assume that $I$
is an interval. Let $J=\{\ell,\ldots, r\}$ be the smallest interval containing $J_a\cup J_b$, and
by symmetry we may assume without loss of generality that the index of interest, $k$, is closer to the left: $k-\ell \le r-k$.

Let $d=\deg V$, note that $d\ge 4$ and even. The iteration will use $j^*=\lfloor 2(r-k)/d\rfloor$ steps.
Define the nested intervals
$$J^j=\{k-j d /2,\ldots, k+jd/2\}\cap J ,\qquad j=1,\ldots, j^*-1$$ and let
$J^{j^*}=J$. Let
$s_j=\sum_{k\in J^j}a_k^2+b_k^2$.  We are interested in $s_1$ and
we will control the $s_j$ recursively from $j^*-1$ to $1$.

Let $J_a^j=J_a\cap J^j$, let $J_b^j=J_b\cap J^j$,
and let $I^j$ be the set of indices at most $d/ 2$ away from $J_a^j\cup J_b^j$.
Corollary \ref{c:l2bound} applied to $J_a^j$, $J_b^j$
and $I^j$ gives
\begin{equation}\sum_{k\in I^j}(a_k^2+b_k^2)\le
c\gamma_4|I^j|+c(1+\gamma_4^{1-d} )\left(\sum_{k\in I^j\setminus J_a^j} a_k^{d}+\sum_{k\in I^j\setminus J_b^j} b_k^{d}\right).
\end{equation}
The left hand side is bounded below by $s_{j}$. On the right hand side, we have boundary terms
$$
\sum_{k\in J^j\setminus J_a^j} a_k^{d}+\sum_{k\in J^j\setminus J_b^j} b_k^{d}\le cy_1^{d}.$$
For $j<j^*$ the rest of the summands have indices from $I^j\setminus J^j$, which is a subset
of $(J_{j+1}\setminus J_j)\cup \partial J$. These
can be bounded by $c\left((s_{j+1}-s_j)^{d/2}+y_1^{d}\right)$.
So for $j<j^*$, we have
\begin{equation}\label{e:sj}
s_j \le c\gamma_4j + c(1+\gamma_4^{1-d}) \left((s_{j+1}-s_j)^{d/2}+y_1^{d}\right).
\end{equation}
For $j=j^*$ only the first kind of boundary terms appear, so we have
\begin{equation}\label{e:sj*}
s_{j^*} \le c\gamma_4j^* + c(1+\gamma_4^{1-d}) y_1^{d}.
\end{equation}
We now proceed to analyze two cases. First assume $y_1\le 1$, $y_2\log(e/y3)\le 1$, and set $\gamma_4=y_1+y_2\log(e/y3)$. Then \eqref{e:sj}, \eqref{e:sj*} simplify to
$$
s_j \le c\gamma_4j + c\gamma_4^{1-d} (s_{j+1}-s_j)^{d/2}, \qquad s_{j^*} \le c \gamma_4j^*.
$$
We use Lemma \ref{l:iteration} with
$$
\gamma_2=c\gamma_4, \quad \gamma_3=2\gamma_2, \quad \gamma_1=c'\gamma_4^{d-1}, \quad x_0=0
$$
to get that either $s_1\le \gamma_3 = 2c\gamma_4$, or we have
\begin{equation}\label{e:sj1}\notag
s_{j^*} \ge \gamma_3(j^*)^{1+1/(d/2-1)}\ge c\gamma_4 (j^*)^{1+1/(d/2-1)}.
\end{equation}
Together with $s_{j^*}<c\gamma_4j^*$ the latter implies $j^*<c$ and so $s_1\le s_j \le c\gamma_4$. The claim follows.

Now assume $y_1>1$ or $y_2\log(e/y_3)>1$. Set $\gamma_4=y_1^{d}+y_2\log(e/y3)$. Then \eqref{e:sj}, \eqref{e:sj*} simplify to
$$
s_j \le c\gamma_4j + c(s_{j+1}-s_j)^{d/2}, \qquad s_{j^*} \le c \gamma_4j^*
$$
we use Lemma \ref{l:iteration} with
$$
\gamma_2=c\gamma_4, \quad \gamma_3=c', \quad \gamma_1=c'', \quad x_0=c'''\gamma_4^{1-2/d}
$$
to get that either $s_1\le c\gamma_4^2$, or we have
$s_{j^*} \ge c(j^*)^{1+1/(d/2-1)}$.
Together with $s_j<c\gamma_4j^*$ the latter implies $j^*<c\gamma_4^{d/2-1}$ and so $s_1\le s_{j^*} \le c\gamma_4^{d/2}$. The claim follows.\end{proof}
\subsection{Lower bounds}

We continue with the setup of the previous subsection with the
additional assumption $\alpha_i\ge 0$ and show that:

\begin{proposition}
\label{prop:lowerbound}
There exists a constant $c_V$ depending on $V$ only so that if for the minimizers $a^\sharp, b^\sharp$ of $\mathcal H$
of equation
\eqref{e:localH2}  together with the boundary conditions satisfy
$$\max_{j\in J\cup \partial J} |a^\sharp_j|+|b^\sharp_j|\le y$$
for some $y\ge 1$, then for all $k\in J$ we have
$$
\log b^\sharp_k \ge -c_V\frac{\alpha_k}{y^{\deg V}}.
$$
\end{proposition}
\begin{proof}
Let $a,b$ equal $a^\sharp, b^\sharp$ except let $b_k=1$. By the minimizer property we have
$$
0\le \mathcal H(a,b)-\mathcal H(a^\sharp, b^\sharp) \le -\alpha_k \log b_k +
\sum_{k\in \pi} w_{\pi}(a,b)- w_{\pi}(a^\sharp, b^\sharp).
$$
The sum is only over path $\pi$ that pass
through $k$, and $w_\pi$ is the monomial corresponding to $\pi$ in the
path expansion of $\tr(V(T_J))$ (all other paths have the same contribution). Counting such paths we get
\[\sum_{k\in \pi} |w_{\pi}(a,b)| \le c'_V(\deg V) 3^{\deg V}y^{\deg V},
\]
and the same holds for $a^\sharp, b^\sharp$. The claim follows.
\end{proof}

\section{Minimizers and boundary conditions}
\label{s:minimizerboundary}

We continue the study of the conditional minimizers of $H$, demonstrating that  they are relatively insensitive to the boundary conditions.
The typical setup now is that we we fix the
values of $a,b$ on a set of indices
(which usually will
be an interval or two intervals $-$ a one or two sided boundary), and minimize
 $H$ subject to these conditions.

Again we will consider such minimizers
for slightly more general Hamiltonians of the form
\begin{equation}\label{e:generalH}
\HH=\HH(a,b)
=\tr(V(T))-\summ_{k=1}^{n-1}\alpha_k\log(b_k)
\end{equation}
with  $\alpha_k\in[0,1]$.   Compared with the modified Hamiltonian \eqref{e:localH2} of the previous section, it is convenient here  to assume the nonnegativity of the coefficients $\alpha_k$ as the conditioning we will need to consider will always have the effect of ``disconnecting" the final stretch of indices.   Recall that in the  actual Hamiltonian  $H$
it is only the indices $k > n - 1/\beta$ (for $\beta < 1$) for which the analogous coefficients are negative.

From now on, for any set of indices $I$, let $\delta I$ be the set of indices outside
$I$ that are at most $\deg V/2$ away from $I$.
An   important consequence of $V$ being polynomial the values of the conditional minimizer $(a_i, b_i)$ for $i\in I$
 depend only on the conditioned values $(a_k, b_k)$ for $k \in \delta I$.

In each of the next two propositions we compare two conditional minimizers of $\mathcal H$ for  same interval $I$, but with respect to different boundary conditions outside $I$. Again, only the conditioned values of the variables in $\partial I$ matter for in problem. To get the desired  bounds though, we requite the additional assumption  that one of the minimizers is in fact a minimizer of for all the variables $I\cup \partial I$ given the variables outside $I\cup \partial I$.

Introducing  the following notation for $\ell^2$-distance over a set of indices
$$
\|(a,b)-(a',b')\|^2_I=\sum_{i\in I} (a_i-a_i')^2+(b_i-b_i')^2,
$$
the warm-up bound reads:

\begin{proposition}\label{prop:close}
Consider any two minimizers $(a, b)$ and $(a', b')$
of $\HH$ over the variables with indices in a set $I$. The variables outside $I$ serve as boundary conditions and are generally different.
Assume further that $(a',b')$ is also a minimizer over variables with indices in $I \cup \partial I$.
Then,
$$
\| (a,b) - (a',b') \|_{I} \le c  \| (a,b) - (a',b')  \|_{\partial I}
$$
as long as $|a_k|, |a^{'}_k|, b_k,  b_k' \le b^*$ for all $k$, and
$b_k, b_k'  \ge  b^*$ for  $k \in I$. The
constant $c$ depends on $V$, and $b_*,b^*$ only.
\end{proposition}

\begin{proof} Without loss of generality, we can assume that $(a,b)_i=(a',b')_i$ for indices outside $I\cup \partial I$.
Let $(a^{\diamond}, b^{\diamond})$ equal $(a',b')$ on $I$ and $(a, b)$
elsewhere. We consider the problem of minimizing $\mathcal H$ over the indices in $I\cap \partial I$ with the variables with different indices fixed.
This problem is solved by $(a',b')$, so by uniform convexity (Lemma \ref{l:UniformConvexity}), the candidate $(a,b)$ satisfies
\begin{eqnarray*}
c_1 \| (a, b) - (a', b') \|^2 &\le& \HH (a, b)-
\HH (a', b').
\end{eqnarray*}
Now consider the problem of minimizing $\mathcal H$ over the indices in $I $, with variables with indices outside $I$ fixed. For the boundary conditions given by $(a,b)$ on $\partial I$, the variables $(a,b)$ solve this problem, and the other candidate $(a^\diamond, b^\diamond)$ satisfies $\HH (a, b)\le \HH(a^\diamond, b^\diamond)$. Thus we get
\begin{eqnarray*}
\HH (a, b)-
\HH (a', b') &\le& \HH ( a^{\diamond},b^{\diamond} )- \HH (a',b')
 \\ &\le&
c_2 \| (a^{\diamond},b^{\diamond}) - (a', b') \|^2\\
&=& c_2 \| (a, b) - (a', b') \|^2_{\partial I}.
\end{eqnarray*}
The last
inequality needs Taylor expansion around $(a', b')$ and
a bound on the second derivative. Note that since $(a',b')$ solves the minimization problem for the variables with indices in $I\cup \partial I$, the linear terms vanish, and we have
$$
\HH (a,b)- \HH (a', b') \le \frac{1}{2} \sup
\| \Hess_{\HH} \|_{2\to2} \| (a,b) - (a', b') \|^2
$$
where the supremum is over $a$'s and $b$'s on the line segment between
$(a,b)$ and $(a', b')$. The second partial derivatives
of the polynomial part of $\HH $ are bounded above by
constant times a power of $b^*$, and that of the
logarithmic part are bounded above by a constant times
$b_*^{-2}$.

Since ${\Hess}_{\HH}$ has at most $2 \deg(V)$ nonzero entries in each row,
and each one is bounded, it follows that as a quadratic form it  is bounded by a
constant times the identity matrix by the Gershgorin Circle
Theorem.
\end{proof}

Iterating the above produces the estimate we will use going forward.

\begin{proposition}\label{prop:minbound}
Consider now any two minimizers $(a, b)$ and $(a', b')$
of $\HH$ over the variables with indices in an interval $I$. The variables outside $I$ serve as boundary conditions and are generally different.
Assume further that $(a',b')$ is also a minimizer over variables with indices in $I \cup \partial I$.
Then, for all $k \in I$,
$$
| a_k - a_k'| + |b_k - b_k' | \le c\exp(-c' \,\dist(k , \partial I))\;  \| (a,b) -(a',b') \|_{\partial I}$$
as long as $|a_k|, |a_k'|, b_k, b_k' \le b^*$ and $\alpha_k\ge \alpha_*>0$ for all $k\in \partial J$. The constants
$ c, c' > 0$ depend on $V$ and $\alpha_*,b^*$ only.
\end{proposition}

\begin{proof}
By Theorem \ref{thm:upper bound} and Proposition \ref{prop:lowerbound} the upper bound $b^*$ on $\partial I$
coupled with the lower bound on  the $\alpha$'s imply upper bounds on $|a_k|, b_k, |a_k'|, b_j' \le {\hat b^*}$ as well as  lower
bounds on $b_k,b_k' \ge b_*$ for $k\in I$. This fact will be required in order to employ
Proposition \ref{prop:close}.

Fix $k$, and divide the interval $I\cup \partial I$ into consecutive blocks $I_{-m},\ldots, I_{m}$, so that
\begin{itemize}
\item  $\partial I = I_{-m}\cup I_m$,
\item for $1\le |j|\le m-1$ the intervals $I_j$ have length $\deg V/2$,
\item $k\in I_0$, and $\dist(I_0^c,k)\le  \deg V/2$.
\end{itemize}
This forces $m\ge \dist(k,\partial I)\times 2/\deg V$. For $j=0,\ldots, m$ set
$$
x_j = \| (a,b) - (a',b')  \|_{I_j \, \cup I_{-j}}^2.
$$
The value of a conditional minimizer at an index $i$ only depends on any conditioned values
for indices within $\deg V/2$ of $i$. This implies that, for any subinterval $I'$ of $I$, the values of $(a,b)$ (or $(a',b')$
are the minimizers of $\HH$ conditioned on the respective values of $(a,b)$ (or $(a',b')$ on $\partial I'$.
Therefore, applying Proposition \ref{prop:close} to the interval $I_{1-j}\cup \cdots \cup I_{j-1}$ we find that
$$
\sum_{i=0}^{j-1} x_i \le c x_j,
$$
with the same constant $c$ for each $j$.
Applying the Gronwall-type inequality of
Lemma~\ref{lem:tobewritten} below yields
$$
 x_0\le c x_m e^{-c' (m+1)}
$$
for some $c'>0$. This implies the statement of the proposition.
\end{proof}

\begin{lemma}\label{lem:tobewritten}
Let $x_i\ge 0$ satisfy $x_1+\ldots +x_k\le cx_k$ for some $c>0$.
Then, $x_1\le c x_ke^{-c' k}$ with another constant  $c'>0$.
\end{lemma}
\begin{proof} Define block sums $y_j=x_{jq}+\ldots
+x_{(j+1)q-1}$ where $q=1/c$. Then $y_1+\ldots +y_\ell\le
y_{\ell+1}$.  Inductively it is easy to see that
$y_{\ell} \ge 2^{\ell-2}y_{1}$, and it follows that
\begin{eqnarray*}
c x_k  & \ge & x_1 + \cdots + x_{k} \\
            & \ge & y_1 + \cdots + y_{[k/q]} \\
            & \ge & y_1 (1 + 2 + \cdots + 2^{[k/q]-2}) \ge  x_1 e^{c' k},
\end{eqnarray*}
as desired.
\end{proof}

\section{Conditional and local minimizers are close}
\label{s:conditionallocal}

We are finally in position to the global (and conditional) minimizers are well approximated by the
the local minimizers.  Recall the local Hamiltonian \ref{localH} introduced in Section \ref{s:local minimizers},
\begin{equation*}
H^{(x)}=H^{(x)}(a,b)=\tr(V(C))-\sum_{k=1}^{n}(1-x)\log(b_k),
\end{equation*}
where again $C$ is s symmetric circulant matrix with main diagonal given by the $a$'s, first off-diagonal given by the $b$'s and zeros elsewhere. Denote again
by $
 (a^\dagger=a^\dagger(x),b^\dagger=b^\dagger(x))
 $ its (independent of $n$, unique) minimizer, Proposition \ref{p:analytic}. Recall as well
 the local minimizers corresponding to index $k$:
\begin{equation*}
  a^\dagger_k=a^\dagger(k/n+1/(n\beta)), \qquad b^\dagger_k=b^\dagger(k/n+1/(n\beta)).
\end{equation*}

\begin{proposition}\label{p:twoH} Let $\mathcal H, \mathcal H'$ be two Hamiltonians given by $\alpha_i, \alpha_i'$, respectively as in \eqref{e:generalH}. Let $J$ be a subinterval of $1,\ldots, n$, let $(a,b)=(a',b')$ outside $J$, and let $(a,b)=(a',b')$  be the minimizers of the corresponding Hamiltonians inside $J$. Then with
$q =\sum_{j\in J} (a_j-a'_j)^2 + (b_j-b'_j)^2$ we have
\begin{align}\label{e:twoH}
q\le \left| \sum_{k \in J} (\alpha_k-\alpha_k')\log b_k \right|+\left| \sum_{k \in J} (\alpha_k-\alpha_k')\log b'_k \right|
 \le c\log ( b^*/b_*)\max_{k \in J}|\alpha_k-\alpha_k'|.
\end{align}
where the second inequality assumes that
\begin{equation}\label{e:bcond2}
b_k,b_k'\in[b_*,b^*]\qquad \mbox{ for }k\in J.
\end{equation} Moreover, for every $y>0$ there exists $b_*,b^*>0$ so that \eqref{e:bcond2} holds as long as $\alpha_k>1/y$ for $k\in J$ and
$|a'_k|,b'_k\le y$ for $k$ with $0<\dist(J,k)\le \deg V/2$.
\end{proposition}

\begin{proof}
By the definition of the minimizer, we have
$$\mathcal H(a,b)\le \mathcal H(a',b') \qquad \mathcal H'(a',b') \le \mathcal H'(a,b)
$$
This implies that the two intervals, defined by their endpoints (whose order we do not specify)
$$
\{\mathcal H(a,b), \mathcal H'(a,b)\} \qquad \mbox{and} \qquad \{\mathcal H(a',b'), \mathcal H'(a',b')\}
$$
overlap, and therefore the difference between any two of these four values is bounded above by the sum of the lengths of these intervals.
$$
\mathcal H(a',b')-\mathcal H(a,b) \le
|\mathcal H(a,b)-\mathcal H'(a,b)|+|\mathcal H(a',b')-\mathcal H'(a',b')|
$$
by convexity, the left hand side is bounded below by a constant multiple of $q$. For the right hand side, all polynomial terms vanish. The remaining terms  are bounded by the right hand side of \eqref{e:twoH}.
\end{proof}

\begin{corollary}\label{c:localtoconditional} Let $J=\{\kappa-\ell,\ldots, \kappa+\ell\}$ be a subinterval of $1,\ldots, n$, and
consider the minimizer $(a,b)$ of $H$ given the values outside $J$. Consider also the modified hamiltonian
\begin{equation}\label{localH3}
H_\kappa(a,b)=V(a,b) - \sum_{k=1}^{n-1} (1-\kappa/n-1/(n\beta))\log b_k
\end{equation}
which differs from $H$ in that the coefficients of the $\log$ terms do not change with $k$. Fix $(a',b')=(a,b)$ outside $J$, and let $(a',b')$ be the minimizer of $H_\kappa$ inside $J$. Assume that $b_k\ge b_*$ for $k\in J$. Let
$$
q =\sum_{j\in J} (a_j-a'_j)^2 + (b_j-b'_j)^2, \qquad y=\max_{i=1\ldots \ell} |b'_{\kappa-i}-b'_{\kappa+i}|
$$
Then we have
$$
q\le c\frac{ \ell^2}{n} \max\left(y, \frac{\ell^2}{n}\right)
$$
for a constant $c = c(b_*)$.
\end{corollary}

\begin{proof}
We use the first bound in  \eqref{e:twoH}. The first sum is bounded above by
\begin{align*}
\left| \sum_{k \in J} \frac{\kappa-k}{n} \log b_k \right|
 = \left| \sum_{i=1}^\ell \frac{i}{n} \log (b_{\kappa+i}-\log b_{\kappa-i})  \right|
 \le \frac{1}{ b_* n} \sum_{i=1}^{\ell} i  \left| b_{\kappa+i} - b_{\kappa-i} \right|,
\end{align*}
and likewise for $(a',b')$, having used the simple inequality $| \hspace {-.2em}\log x - \log y  |  \le \frac{|x-y|}{ x \wedge y}$.

Combining the above gives
\begin{align*}
 q &  \le \frac{2}{b^* n} \sum_{i=1}^{\ell} i | b_{\kappa+i}' - b_{\kappa-i}' | +
                        \frac{1}{ b^* n} \sum_{i=1}^{\ell} i ( | b_{\kappa+i}-b_{\kappa+i}' | + |b_{\kappa-i} - b_{\kappa -i}'| ) \\
                    & \le  \frac{ c \ell^2 }{ n} y +
                    \frac{c  \ell^{3/2}}{ n} \sqrt{q},
\end{align*}
after an application of Cauchy-Schwarz. In the case $y\ge \sqrt{q}$ we get
$
q\le  2c\frac{  \ell^{2}}{ n} y.
$
When $y<\sqrt q$ we get
$
q\le  2c\frac{  \ell^{2}}{ n} \sqrt{q},
$
and so
\[
q\le  4c^2\frac{  \ell^{4}}{ n^2} \le 4c^2 \frac{  \ell^{2}}{ n}\max\left(y, \frac{\ell^2}{n}\right).
\qedhere\]
\end{proof}

\begin{proposition}\label{p:cond loc prelim}
Let $J$ be a subinterval of $1\ldots n-\lfloor \eps n\rfloor$, and let $\kappa$ be its midpoint. For any set of indices $I$, define
$$
\delta_I = \max_{i\in I} |a_j-a^\dagger_\kappa|+|b_j-b^\dagger_\kappa|.
$$
Let $(a,b)$ denote the minimizers of $H$ for the variables with indices in $J$ with some boundary conditions on $\partial J$.  Then for any $j\in J$ we have
$$
|a_j-a_\kappa^\dagger|+|b_j-b_\kappa^\dagger| \le c\,\max\left(\delta_{\partial J} \exp(-c_1\dist(j, \partial J)), \frac{\ell^2}{n}\right).
$$
where $c,c_1$ depend on $V$, $\beta$ and $\eps$ only.
\end{proposition}

\begin{proof}
First note that Theorem \ref{thm:upper bound} and Proposition \ref{prop:lowerbound}  give constant upper and lower bounds
on minimizers in terms of $\eps, V, \beta$ and the boundary values.
When we invoke Corollary \ref{c:conditional-local}, we will implicitly use these bounds. It is important that
in the repeated use of Corollary \ref{c:conditional-local} we can use the same bounds, and
these don't have to be iterated.

By induction, it suffices to prove that
$
\delta_J\le c\max(\delta_{\partial J}/2,\ell^2/n)
$
and that there exist a constant $\ell'$, so that if $J'$ is  interval $J$ reduced by $\ell'$ on both sides, then
$
\delta_{J'}\le \max(\delta_{\partial J}/2,c\ell^2/n)
$.
Corollary \ref{c:localtoconditional} (in its notation) says that with
$$
q =\sum_{j\in J} (a_j-a'_j)^2 + (b_j-b'_j)^2, \qquad y=\max_{i=1\ldots \ell} |b'_{\kappa-i}-b'_{\kappa+i}|,
$$
it holds that
\begin{equation}
\label{loctocondagain}
q\le c\frac{ \ell^2}{n} \max\left(y, \frac{\ell^2}{n}\right)
\end{equation}
The values $a',b'$ are minimizers of the local Hamiltonian \eqref{localH} with some fixed boundary conditions. Note also that the constant function $a^\dagger_\kappa,b^\dagger_\kappa$ is also a minimizer of the same Hamiltonian on any interval where boundary conditions are constant also given by $a^\dagger_\kappa,b^\dagger_\kappa$ (this follows from the definition \eqref{bdagger} that says that they are minimizers of the periodic problem). So we can apply Proposition \ref{prop:minbound} to get
\begin{eqnarray}
\label{expdecay}
|a'_j-a^\dagger_\kappa|+|b'_j-b^\dagger_\kappa| &\le&
c\exp(-c_1 \,\dist(j,J^c)) \left( \max_{i\in \partial J}|a'_i-a^\dagger_\kappa|+|b'_i-b^\dagger_\kappa| \right)  \\
&=&
c\exp(-c_1 \,\dist(j,J^c))\delta_{\partial J} \nonumber
\end{eqnarray}
the last equality follows since $a',b'$ agree with $a,b$ on $\partial J_k$. From here, since $(a^{\dagger}_{\kappa}, b_\kappa^{\dagger})$ is constant we have that
\begin{equation}
\label{ybound}
y\le 2c\delta_{\partial J}
\end{equation}
Next, by the triangle inequality and \eqref{expdecay}, we also have that
$$
|a_j-a^\dagger_\kappa|+|b_j-b^\dagger_\kappa| \le c\exp(-c_1 \,\dist(j,J^c))\delta_{\partial J} + 2\sqrt{q}.
$$
Now using   \eqref{ybound} in \eqref{loctocondagain}  produces
$$
2\sqrt{q} \le 2\sqrt{c\frac{\ell^2}{n}\max(\delta_{\partial J},\frac{\ell^2}{n})}
=\frac{1}{4}\sqrt{64 c\frac{\ell^2}{n} \max(\delta_{\partial J},\frac{\ell^2}{n})}
\le \frac{1}{4}\max(\delta_{\partial J},64 c\frac{\ell^2}{n}),
$$
which when substituted  in the previous display gives
$$
|a_j-a^\dagger_\kappa|+|b_j-b^\dagger_\kappa|\le  \max(\delta_{\partial J},\frac{c'\ell^2}{n}) \left(c\exp(-c_1 \,\dist(j,J^c))+1/4\right).
$$
We choose $\ell'$ so that  $c\exp(-c_1\ell')<1/4$,
and the two claims follow.
\end{proof}

\begin{corollary}\label{c:conditional-local}
Let again $J$ be any subinterval of $1\ldots n$. Let $(a,b)$ denote the minimizers of $H$ for the variables with indices in $J$ with some boundary conditions on $\partial J$.  For any $j \in J$ it holds that
$$
|a_j-a_j^\dagger|+|b_j-b_j^\dagger| \le c \,\max\left(\delta_{\partial J} \exp(-c_1\dist(j, \partial J)), \frac{(\log n)^2}{n}\right).
$$
where again $\delta_{\partial J} = \max_{j \in \partial J} |a_j-a_j^\dagger|+|b_j-b_j^\dagger|$.
\end{corollary}

\begin{proof}
First note that if $|J| \le c \log n$ one simply restates Proposition \ref{p:cond loc prelim} using that
$$
\| (a^{\dagger}_{\kappa}, b^{\dagger}_{\kappa})  -
(a^{\dagger}_{j}, b^{\dagger}_{j}) \|_J \le c' \frac{\log n }{n},
$$
which holds since by Proposition \ref{p:analytic}, the functions $x \mapsto (a^{\dagger}(x), b^{\dagger}(x))$ are analytic for $x \in (-\infty,1)$.

If now $|J| > c \log n $ and $\dist(j, \partial J) \le \log n$ we can again apply Proposition \ref{p:cond loc prelim} by taking $j$ in the role of $\kappa$. That is, $j$ can be made
the midpoint of $J' \subset J$ with $|J'| = O(\log n)$ and the boundary conditions on $\partial J'$ just equal to the values of $(a, b)$ there. The point being that, as noted before,
 $(a,b)$ is still the conditional minimizer on $J'$ subject to these boundary conditions.

Finally, consider the case that $|J| > c \log n $, but $j$ is within $\log n $ distance of $\partial J$.  Assume say that $j$ is closer to the left edge of $J$.
By moving in order $\log n$ steps from the left boundary we will find $\deg(V)/2$ stretch of $(a,b)$ which already satisfies $|a_j - a_j^{\dagger}| + |b_j - b_j^{\dagger}|$
$\le c \frac{(\log n)^2}{n}$. The Proposition \ref{p:cond loc prelim} can then be  applied yet again to the subinterval $J'$  defined by the shared boundary of $J$ to its left and the just identified good stretch of coordinates to its right. This will produce the type of statement desired for the midpoint $\kappa$ of $J'$ and with $\delta_{\partial J'}$ in place of $\delta_{\partial J}$.   For the first issue use again that $(a_j^{\dagger}, b_j^{\dagger})$ and $(a_\kappa^{\dagger}, b_{\kappa}^{\dagger})$ are close throughout $J'$. For the second, easily $\delta_{\partial J'}
\le \delta_{\partial J}$.
\end{proof}

\section{Bounding the field}
\label{s:boundingthefield}

As in the case of the minimizers, good concentration properties of the field of random variable $(A,B)$ begins by showing some simple upper and lower bounds hold, now with high probability.  The convexity of the Hamiltonian $H$ plays a key role again here: the law $e^{-n \beta H}$ expected to place most of its mass in a neighborhood of minimizers, for which we have sharp bounds.  We are able to get by with a relatively simple Gaussian domination.

\subsection{A Gaussian lemma}

The following may be viewed as instance of the Brascamp-Lieb inequality \cite{BL}. The short proof is included as we were unable to locate a statement in precisely the form required, allowing the ``center" of the log-concave measure $e^{-f}$ below to take place on the boundary of the domain in question.

\begin{lemma}
\label{l:GaussianLem}
Let $A$ be a convex subset of $\mathbb R^n$. Let $f:A\to \mathbb R $ be a
convex with its minimum on $A$ achieved at $y\in A$ (this
may be on the boundary). Assume that $f$ has Hessian
satisfying $Hf(x)\ge uI$ for all $x\in A$. Let $X$ be a
random variable with density proportional to $e^{-f}$. Let
$G$ be the Gaussian vector with density proportional to
$e^{-u\|x\|^2/2}$. Then $\|X-y\|$ is stochastically dominated
by $\|G\|$.
\end{lemma}

\begin{proof} Without loss of generality, assume $u=1$ and $y=0$. Fix
any $0\le s<t$ and $\omega \in \S^{n-1}$ so that $\omega t\in A$.
Then $\frac{d^2}{dt^2}f(t\omega) = \omega^t Hf(t\omega)\omega \ge
1$, from which it follows that $\frac{d}{dt}f(t\omega) \ge t$,
since the derivative is nonnegative at $t=0$, because $0$ is a minimum of
$f$ (it may  be positive if $y$ is on the boundary). Thus,
\[ f(t\omega) =  f(s\omega)+ \int_s^t \frac{d}{du} f(u\omega) du \ge  f(s\omega)+\frac{t^2}{2}-\frac{s^2}{2}. \]
Hence,
\[ e^{-f(s\omega)}e^{-t^2/2}\ge e^{-f(t\omega)}e^{-s^2/2}. \]
This also holds for $s\omega$ or $t\omega$ outside $A$ with the
convention that in that case $e^{-f(s\omega)}=0$ or
$e^{-f(t\omega)}=0$, respectively. Fix $r>0$. Multiply by
$s^{n-1}t^{n-1}$ and integrate over $t>r$, $s<r$ and over $\omega
\in S^{n-1}$ to get
\[   \int_{\|x\|>r} e^{-f(x)} dx
  \int_0^r e^{-s^2/2}s^{n-1}ds   \le
   \int_{\|x\|<r} e^{-f(x)}dx
  \int_r^{\infty} e^{-t^2/2}t^{n-1}dt  .
\]
Multiplication by the appropriate normalization constants
yields
\[ \P[\|X\|>r]\,\P[\|G\|<r] \le \P[\|X\|<r]\,\P[\|G\|>r].
\]
We can rewrite this as
$$
\P[\|X\|>r] \,(1-\P[\|G\|>r] )\le (1-\P[\|X\|>r]) \,\P[\|G\|>r],
$$
from which we get the desired domination
$\P[\|G\|>r] \ge \P[\|X\|>r]$.
\end{proof}

\subsection{Upper and lower bounds on $(A,B)$ }

The Gaussian Lemma (Lemma \ref{l:GaussianLem}) coupled with the bounds on the minimizers from Section \ref{sec:minimizerbounds}
imply that, with high probability, the random variables $(|A|, B)$ themselves are bounded above and the $B$'s are bounded below (with a small caveats depending on $\beta$).

Here we return to the original tridiagonal matrix entries for the Dyson $\beta$-ensemble with Hamiltonian
$$
   H(a,b)  = V(a,b) - \sum_{k=1}^{n-1} ( 1 - k/n - 1/(n \beta) ) \log b_k.
$$

The upper bound reads:

\begin{proposition}[Upper bound on the field]
\label{prop:cheapbound}
Let $A,B$ denote the random tridiagonal matrix entries for the Dyson $\beta$-ensemble.
There exists constants $c_1,c_2,c_3>0$ depending on $V$ and $\beta$ so that for all $\beta \ge 1$ and
all $n$ we have
  \begin{equation}\label{e:cheap1}
  \P\left(\sum_{k=1}^n (A_k-a^\circ_k)^2 + \sum_{k=1}^{n-1} (B_k-b^\circ_k)^2
  >c_1\right) \le e^{-c_2n}
  \end{equation}
here $a^\circ, b^\circ$ are the minimizers of the corresponding Hamiltonian $H$.

For all $\beta >0$  we also have
  \begin{equation}\label{e:cheap2}
  \P(|A_1|,\ldots, |A_n|\le c_3,\mbox{ and } B_1,\ldots B_{n-1} \le c_3) \ge 1-e^{-c_2n}.
  \end{equation}
Moreover, for all $\beta>0$ and $\eps>0$ there is $\delta>0$ so that
for all
$n$ and $s\ge 1-\delta$ we have
  \begin{eqnarray}\notag
  &\P\Big(|A_k-a^\dagger(1)|,B_k\le \eps \mbox{ for all }k\ge sn \,\Big|
  \,|A_k-a^\dagger(1)|,B_k\le \delta \mbox{ for all }k\in [sn,sn+\deg V] \Big) \\ & \qquad \ge 1-e^{-c_2n}. \label{e:cheap3}
  \end{eqnarray}
\end{proposition}

\begin{proof}
We first treat the $\beta\ge 1$ case.

The Hamiltionian corresponding to the distribution is uniformly
convex with Hessian bounded below by $c_uI$ by Lemma \ref{l:UniformConvexity}. Thus the Gaussian Lemma
(Lemma \ref{l:GaussianLem}) implies that $\|(A-a^\circ,B-b^\circ) \|^2$
 is stochastically dominated by the norm an
i.i.d.\ Gaussian vector $G$ in dimension $ n$ and
entry variance  $c_2^2=1/(2\beta c_un)$.

Noting that $\E \|G \|^2 = c_2^2$ and that $G/c_2$ is the average of $n$
 independent
$\chi_1^2$ random variables $\P (G/c_2 > c_1) \le  e^{-nI(c_1)}$ for any
$c_1 > 1$ and $I(\cdot)$ the rate function for the $\chi^2$-distribution.
For the second claim \eqref{e:cheap2} recall that Theorem \ref{thm:upper bound}  proves that (conditional) minimizers
satisfy a uniform bound. For \eqref{e:cheap3}, use the fact that this bound tends to zero
as the coefficients of the logarithmic terms do, and apply previous argument to the conditional
distribution.

In the $\beta<1$ case the the $\log$ terms with positive
coefficients destroy the convexity of the Hamiltonian $H$. Fix a value $y$ so that for $b\ge y$ the
second derivative $-\frac{1}{\beta} \log b$ is at most half the uniform convexity constant $c_u$ from Lemma \ref{l:UniformConvexity}.

We therefore condition on the set of indices $k>n-1/\beta$ so that $B_k<y$, as well as
the values of these $B_k$. The entries minimizer of the conditional Hamiltonian are then bounded
by a constant (depending on $V$ only) by Theorem \ref{thm:upper bound} (this requires the smallest
coefficient $-1/{\beta n}\ge-1$, so it holds for large enough $n$; for small values of $n$ any bound works).

The conditional Hamiltonian
is uniformly convex: the log terms with $b_k<y$ are removed; for $b_k\ge y$ the second derivatives
of $\log b_k$ terms are dominated by the uniform convexity constant $c_u$. We can then bound the conditional distribution
of the rest of the $A,B$ the
the same way as in the $\beta>1$ case.  Averaging over the conditioning gives \eqref{e:cheap2}.

To get \eqref{e:cheap3}, we can let $y\to 0$ with $\eps$, and apply this previous argument
with the quantitative version of Theorem \ref{thm:upper bound}.
\end{proof}

For a lower bound we have:

\begin{proposition}[Lower bound on the field]
\label{prop:lowerfield}
For any $k \le n - 1/\beta$, there are constants $c_1, c_2, c_3$ such that
$$
   \P (  B_k > e^{- c_1 \frac{n}{n-k}},   |A_1|,\ldots, |A_n|\le c_3 \mbox{ and } B_1,\ldots B_{n-1} \le c_3) \ge 1-e^{-c_2 n}.
$$
Here $c_1$ must be chosen to depend on $c_3$. Otherwise the $c_1, c_2, c_3$ depend only on $V$ and $\beta$.
\end{proposition}

\begin{proof} Consider $(a, b)$ and $(a', b')$ which agree everywhere except at a given index $k \le n - \frac{1}{\beta} $ where it holds $b_k < b_k'$. By the same type of considerations
used in Proposition \ref{prop:lowerbound}  we have that, assuming all $|a_j|, b_j \le b^*$,
$$
 \exp \{-n \beta H(a, b) + n \beta H(a', b') \}   \le \exp\{n c + (\beta(n-k)-1) \log(b_k/b_k') \},
$$
with a $c = c(b^*, \beta, V)$.  We tacitly assume $b^* \ge 1$.
The inequality is written in this way to emphasize that the left hand side is the ratio of the densities at $(a,b)$ and $(a', b')$.  If then
$ \log (b_k/b_k') <  \rho := - \frac{2c}{ 1 - k/n - 1/n\beta} < 0$ this ratio is bounded by $e^{-nc}$, still granted the overall bound on the $|a|$'s and the $b$'s.
Choosing for example $b_k < e^{\rho}$ and $b_k' = b_k+1$ implies that
$$
   \P  (  b_k < e^\rho \, \vert \, a_j   \mbox{ and }  b_j  \mbox{ for }  j \neq  k )  \one_{ |a_j|, | b_j| \le b^* } \le e^{-nc}.
$$
Taking expectations and slightly adjusting $\rho$ yields the claim.
\end{proof}

\section{Gaussian approximation of the field}
\label{s:Gaussian}

We  now show that, for suitable ranges of the indices,
the random variables $(A,B)$ are well approximated by a Gaussian process,
or really by a mixture of Gaussian processes.
The precise statement is given in Proposition \ref{p:variation} below. This is a first pass at the
functional central limit theorem which identifies the Brownian potential in the Stochastic Airy Operator,

First we record two more concentration estimates for the field about various minimizers of the Hamiltonian $H$. The second is a refinement of the first. Both are consequences of the Gaussian Lemma (Lemma \ref{l:GaussianLem}).

\subsection{Concentration about the minimizers}

 Fix $0<b_*<b^*$.
We say that $G_i(a,b)$ holds if we have $|a_i|,b_i\le b^*$
and $b_i>b_*$. For a set of indices $J\subset\{1,\ldots, n\}$, we say that $G_J(a,b)$ holds if $G_i(a,b)$ holds for all $i\in J$; we will say $G(a,b)$ holds if $G_i(a,b)$ holds for all indices.

For an interval $I$, let $\partial I$ denote union of the stretches of length $(\deg V)/2$
immediately before and after $I$.

\begin{lemma}[Concentration in short intervals]\label{lem:concentration short}
Fix $\delta>0$. Let $I$ be a subinterval of $\{1 ,\ldots, n
\}$ of length at most $c_1\log n$.

Consider random variables
$(A_i,B_i)_{i\in I}$  picked from the distribution \eqref{eq:matrixpdf}, where the
values $a,b$ on $\partial I$ are fixed. Let $(a^\sharp,b^\sharp)$ denote the minimizer of the density given
the boundary conditions $\partial I$. Then there is a $c=c(\delta,b_*,b^*)$ so that for all $n$ we have
\begin{equation}\label{eq:ldiv}
 \P \left(\sum_{i\in I}
(A_i-a_i^{\sharp})^2+(B_i-b_i^{\sharp})^2>n^{2\delta-1}\right)\le c
e^{-n^\delta}.
\end{equation}
\end{lemma}
\begin{proof}
The log density of the conditional distribution is uniformly
convex  by Lemma \ref{l:UniformConvexity}. Thus the Gaussian Lemma (Lemma \ref{l:GaussianLem}) implies that $|B -
b_k|$ is stochastically dominated by the norm an
i.i.d.\ Gaussian vector in dimension $2c_1\log n$ and
entry variance  $c_2/n$, where $c_2$ depends on the uniform
convexity constant only.

The square of a 2-dimensional Gaussian has exponential distribution. This gives the (standard) large deviation bound that for the sum of squares of
independent standard Gaussians $X_i$
$$
\P(X_1^2+\ldots + X_m^2>a) \le m\P(X_1^2+X_2^2>a/(m/2))
=m\exp(-a/m)
$$
Finally, Gaussian scaling gives for \eqref{eq:ldiv} the upper bound
\[ (2\log n)\exp(-n^{2\delta}/(2c_2c_1\log n))
\le c e^{-n^\delta}. \qedhere
\]

\end{proof}

\begin{lemma}[Concentration everywhere]\label{lem:concentration}
Fix $\delta>0$. Let $I$ be any subinterval of $\{1\ldots (1-\eps)n\}$, Consider random variables
$(A_i,B_i)_{i\in I}$  picked from the distribution \eqref{eq:matrixpdf}, where the
values $a,b$ on $\partial I$ are fixed, and assume that
\begin{equation*}
|a_i-a_i^{\dag}|+|b_i-b_i^{\dag}|\le n^{\delta-1/2} \qquad \mbox{for all } i\in \partial I.
\end{equation*}
where $(a_i^\dag, b_i^\dag)$ are the local minimizers defined in \eqref{bdagger}.
Then there is a $c=c(\eps,\delta,b_*,b^*)$ so that for all $n$ and $I$ we have
\begin{equation}\label{eq:ldiv3}
 \P \left(\exists k\in I : |A_k-a_k^{\dag}|+|B_k-b_k^{\dag}|>n^{\delta-1/2}\right)\le c
e^{-n^\delta}.
\end{equation}
\end{lemma}

\begin{proof}
{\bf Case 1.} First we show that with high probability
$$
|A_k-a_k^{\dag}|+|B_k-b_k^{\dag}|\le n^{\delta-1/2}
$$
holds for indices $k$ satisfying $\dist(k,\delta I)\ge c \log n$.
Let $J$ be a subinterval of $I$ of length  $1+2\lfloor c_1 \log n\rfloor$ centered at $k$.  Fix the values of $A,B$ to be  $(a,b)$, $\partial J_k$ (if  $\partial J$ intersects $\partial I$, then then these values are already fixed by the boundary condition). Let $(A^{(k)}_i, B^{(k)}_i)_{i\in J_k}$ denote the minimizer of the conditional density of the variables given these (random) boundary conditions. Since $k$ is far enough from the boundary of $J$, this minimizer is close to the local minimizers. Indeed, Corollary \ref{c:conditional-local}
applied to the interval $J$ gives that on $G(A,B)_{\partial J}$ we have
$$
|A_k^{(k)}-a_k^{\dag}|+|B_k^{(k)}-b_k^{\dag}|\le n^{\delta-1/2}
$$
now Lemma \ref{lem:concentration short} applied to the conditional distribution of $(A^{(k)}_i, B^{(k)}_i)_{i\in J_k}$ gives that with probability at least $1-ce^{-n^\delta}$ we have
$$
|A_k^{(k)}-A_k|+|B_k^{(k)}-B_k|\le n^{\delta-1/2}
$$
together with the union bound, after increasing $\delta$ a bit, this gives the first case.

{\bf Case 2.} Now consider indices $k$ satisfying $\dist(k,\delta I)\ge c \log n$. When  $|I|<2c_1\log n$, all indices are like that, so we consider this case first. Here, Corollary \ref{c:conditional-local} shows that the maximizers $a^\sharp,b^\sharp$ of
the conditional density of $A,B$ are $cn^{\delta-1/2}$-close to the
local minimizers, and the Gaussian Lemma \ref{l:GaussianLem} shows that the values are concentrated around $a^\sharp,b^\sharp$, and the claim follows after increasing $\delta$ a bit.

{\bf Case 3.}
Finally, assume that $|I|>2c_1\log_n$. Assume that $k$ is closer than $c_1\log n$ to the left boundary  $i_1$ of $I$. Consider the $\deg(V)/2$ indices starting at $i_1+c_1\log n$. From the first case, it follows that $A,B$ at these indices are with high probability $cn^{\delta-1/2}$-close
to the local minimizers. Now apply the second case to the conditional
distribution of $A,B$ on the interval $i_1,\ldots, i_1+c_1 \log n$ to get the desired claim.
\end{proof}

\subsection{(A,B) as a mixture of Gaussians}

Now let $(A,B)$ are picked from the Dyson $\beta$ ensemble distribution \eqref{eq:matrixpdf}, and let $I\subset 1,\ldots (1-\eps)n$ be an interval. The distribution of $A_i,B_i,i\in I$ depends only on the values (denote by $q$) of $A,B$ with indices in $\partial I$. Let $\mu_q$ denote this conditional distribution. Let
$a^\scond, b^\scond$ denote maximizers of the density of $\mu_q$ on $I$, and let them correspond to $q$ on $\partial I$.

Let $\nu_q$ denote the Gaussian measure on $a_i,b_i,i\in I$
with mean $(a^{\scond}, b^{\scond})$ and inverse covariance matrix $n\beta \Hess_{H}(a^\scond, b^\scond)$.

Let $\mu_I$ denote the averaging of $\mu_Q$ with respect to the random $Q$ picked from its marginal distribution. Note that $\mu_I$ is just the marginal distribution of $(A_i,B_i; i\in I)$. Let $\mathcal Q$ denote the set of boundary conditions $(a_i,b_i), i \in \partial I $ so that
$$
|a_i-a_i^{\dag}|+|b_i-b_i^{\dag}|\le c n^{\delta-1/2} \qquad \mbox{for all } i\in \partial I.
$$
Pick $q_0 \in \mathcal Q$, and let $Q'=Q$ when $Q\in \mathcal Q$ and $Q'=q_0$ otherwise. Let $\nu_I=\ev \nu_{Q'}$.

\begin{proposition}\label{p:variation}
Let $I\subset\{ c\log n, \ldots, (1-\eps)n\}$ be an interval of length at most $n^{1/2-\eps}$.
Then as $n\to\infty$ we have
$$\dist_{TV}\left(\mu_I, \nu_I \right) =o(1).
$$
More precisely, we can choose the constant $c$ so that the event
$$
   S = \{   |A_i |  \le  c , \frac{1}{c} \le B_i  \le c \mbox{ for } i \in 1, \dots n(1-\epsilon) \cap  \max_{ i \in I} | A_i -a_i^{(Q)}| + |B_i - b_i^{(Q)}| < c n^{\delta - \frac{1}{2} } \}
$$
satisfies $\mu(S) \ge 1-ce^{-n^{\delta'}/c}$ as well as $ \frac{d \nu_I}{d \mu_I} \vert_S = 1+o(1)$ so long as $\delta < \epsilon/3$.
Moreover, on the event $S$,
\begin{equation}
\label{e:Sbound}
|a_j^\scond -a_j^\dagger| +|b_j^\scond -b_j^\dagger| \le c \max\left( n^{\delta-1/2} \exp(-c_1\dist(j, \partial I)), \frac{(\log n)^2}{n}\right),
\end{equation}
for all $j \in I$.
\end{proposition}

\begin{proof}
Let $H_q$ be the Hamiltonian corresponding to $\mu_q$, so that
$$
\mu_q=\frac{1}{Z_{\mu_q}}\exp(-n \beta H_q)\,da\,db
$$
Let $\Hess_H$ be the Hessian of $H_q$ evaluated at its minimizer
$a^\scond, b^\scond$. We compare the probability measures $\mu_q$ and $\nu_q$,  which can be written as
$$
\nu_q=\frac{1}{Z_{\nu_q}}\exp\left(-\frac{n\beta}2\left\langle
(a-a^\scond,b-b^\scond),\Hess_H(a-a^\scond,b-b^\scond)\right\rangle\right) \,da\,db.
$$
Note that by adjusting the $Z_{
\nu_q}$ we can assume that $H_q(a^\scond,b^\scond)=0$. Then by Taylor expansion
\begin{align*}
\label{e:xi}
|H - \left\langle
(a-a^\scond,b-b^\scond),
\frac{1}{2}   \Hess_H(a-a^\scond,b-b^\scond) \right\rangle |
 \le
 \xi\sum_{i=1}^k |a_i-a_i^\scond|^3+|b_i-b_i^\scond|^3
\end{align*}
where
\begin{equation}
\label{e:xi1}
\xi(a,b)=c\max_{i\in I,t\in [0,1]} \left(\left\vert \frac{\partial^3 H_{q}}{\partial ^3
{a_i}}\right\vert_{a_i(t)}+\left\vert\frac{\partial^3 H_{q}}{\partial ^3
{b_i}}\right\vert_{b_i(t)}\right)
\end{equation}
and $a_i(t)=ta_i+(1-t)a^\scond_i$, and $b_i(t)=tb_i+(1-t)b^\scond_i$.
Note that the bounded-range interaction implies that there
are only a linear number of non-zero mixed third partial
derivatives, and they can be bounded this way.

Next define
$$
 S_\gamma=\left\{(a,b)\,:\, n \,
\xi(a,b)\left( \sum_{i=1}^k |a_i-a_i^\scond|^3+|b_i-b_i^\scond|^3
 \right) \le \gamma \right\}.
 $$
 By \eqref{e:xi1} and the preceding display, on the event $
  S_\gamma$, the Radon-Nikodym derivatives between the measures
  $\mu_q$ and $\nu_q$  lie in the interval $
{Z_{\mu_q}}{Z_{\nu_q}^{-1}}[e^{-\gamma},e^{\gamma}] $.
Further setting,
$$
e^{-\kappa}:=\min(\mu_q(S_\gamma),\nu_q(S_\gamma)),
$$
 we have that,
$$
e^{-\kappa} \le \nu_q(S_\gamma) \le
\frac{Z_{\mu_q}}{Z_{\nu_q}} \int_{S_\gamma} e^\gamma d\mu_q,
$$
from which it follows that
$$
\frac{Z_{\mu_q}}{Z_{\nu_q}}  \ge e^{-\gamma-\kappa}, \mbox{
\ \ \ and by symmetry \ \ \ } \frac{Z_{\nu_q}}{Z_{\mu_q}}
\ge e^{-\gamma-\kappa}.
$$
Thus $\frac{ d \mu_q}{d \nu_q}$ is in fact in the
interval $ [e^{-2\gamma-\kappa},e^{2\gamma+\kappa}].$  This
implies that for any event $D$
\begin{eqnarray*}\mu_q(D) -\nu_q(D) &\le&
\mu_q(D)-e^{-2\gamma-\kappa}\mu_q(D)+\nu_q(S^c_\gamma)
\\&\le& 2(\gamma+\mu_q(S_{\gamma}^c)+\nu_q(S_{\gamma}^c)),
\end{eqnarray*}
using simply that $1-e^{-2 \gamma} \le 2 \gamma$ and the definition of $e^{-\kappa}$.

By symmetry and the last formula we have that
$$
 \dist_{TV}(\mu_q,\nu_q)\le
 2(\gamma+\mu_q(S_{\gamma}^c)+\nu_q(S_{\gamma}^c)).
$$
Now recall the measure $\nu_I$: it is the mixture of $\nu_q$
given by $\ev \nu_{Q'}$. The final form of the total variation bound
the reads,
\begin{eqnarray}
\label{e:tvbound}
 \dist_{TV}(\mu_I, \nu_I)
 &\le & \ev \,\dist_{TV}(\mu_{Q},\nu_{Q'})
 \\&\le&
 2\gamma + 2\ev \mu_Q(S^c_\gamma) + 2\ev
\nu_Q(S^c_\gamma)+P(Q\notin \mathcal
 Q), \nonumber
\end{eqnarray}
where the expectation is with respect to the random values
assigned to $Q$.

We  wish to show the last  three quantities in \eqref{e:tvbound} tend to zero for
a $\gamma=o(1)$. To begin, introduce the event
$$
E = \{ |a_i|, b_i \le b^* \mbox{ and } |b_i| > b_* \mbox{ for all } i
 \in 1, \dots, (1-\epsilon)n  \}.
$$
By Proposition \ref{prop:cheapbound} we have that the indicated upper bounds on the field
hold with high probability.  Lemma \ref{prop:lowerfield} gives the desired
conclusion for the lower bound as we have restricted to indices less than
$(1-\epsilon)n$.

Working on $E$ we can show that the left portion $L$ of
$\partial I$ satisfies $\mathcal Q$
with high probability (the argument for the right part of $\partial I$ is the same).
Consider an interval of $I_L$ length $c \log n$ containing $L$.
Again on $E$, the
events
$$
    \max_{k \in L} | A_k - A_k^{\sharp} | + |B_k - B_k^{\sharp}|
     \le c n^{\delta - \frac{1}{2}},
$$
and
$$
    \max_{k \in L} | A_k^{\sharp} - a_k^{\dagger} | +
       |B_k^{\sharp} - b_k^{\dagger} | \le c \frac{(\log n)^2}{n},
$$
both hold with high probability.  Here $\{A_k^\sharp, B_k^\sharp\}$ are the conditional
minimizers of $H$ subject to the (random) boundary conditions on $\partial I_L$
(drawn from the basic law $\mu$). The first of these holds by Lemma
\ref{lem:concentration short}.
The second follows from Corollary \ref{c:conditional-local}.
 The point is that, save for an event of
exponentially small
probability, $(A_k, B_k)$ are within $O(n^{-1/2+\delta})$ of the local
minimizers throughout $L$.  That is,
$\P( Q \notin \mathcal Q) \rightarrow 0$.

To control $\ev \mu_Q(S^c_\gamma)$ it is enough to consider
$\mu_q (S_{\gamma}^c E)$ for $q \in \mathcal Q$. We have only just shown
that on
$E$  the boundaries conditions may be assumed to satisfy $\mathcal Q$.
Now, also on $E$, we have that $ \xi(A,B) < c'$, recall \eqref{e:xi1}. The upper bounds on
the field and conditional minimizers (Proposition \ref{prop:cheapbound} and Theorem \ref{thm:upper bound})
bound the third derivatives of the polynomial part; the lower bounds from
Propositions \ref{prop:lowerfield} and \ref{prop:lowerbound} will bound the derivatives
of the logarithms. Finally, Lemma \ref{lem:concentration} implies that
$\mu_q (\max_{k\in I} |A_k -a^\scond| + |B_k - b^\scond| >  n^{\delta -1/2})
\le c e^{-n^{\delta}/c}$ where $c$ is independent of $q$. But this estimate
entails a similar bound on $\mu_q(S_{\gamma}^c E)$ by adjusting $\delta$
so that $\delta < \epsilon/3$.

The Gaussian calculation $ \sup_{q\in \mathcal Q} \nu_q(S_\gamma^c)\to 0$
is more immediate as under this law the centered variables each have variance bounded by
$c/n$ (with again $c$ independent of $q$).

The more refined statement $  \frac{d \nu_I}{d \mu_I} \vert_S = 1+o(1)$ (with $\mu(S) = 1 - o(1)$)
follows readily after noticing that the
advertised $S$ is contained in  $E \cap S_\gamma(Q)$ for a $\gamma$ chosen to be $o(1)$.
The estimate \eqref{e:Sbound} simply puts together the outcome of  Corollary \ref{c:conditional-local} with the assumption that
one is working on the event $S$.
\end{proof}

\section{A non-universal operator limit in the global scaling}
\label{s:firststretch}

With the results of the last section everything is in place to prove our main result, at least as applied to the truncated operator $T_{[c \log n , n]}$. Adapting this truncation is not a matter of convenience: the first order $\log n$ entries of the matrix $T_n$ behave very
differently from the quadratic case. Even their first-order  behavior is
different from the entries that come after $\log n$ steps.  We will show that at on one hand the top of $T_n$, on a global scale, encodes the equilibrium measure $\mu_V$ and is thus
non-universal.  On the other hand, Theorem \ref{t:beginning} below shows that the effect of this global non-universal behavior for the first stretch of matrix entries can be absorbed into a suitably small perturbation that does not harm the universality of  the edge scaling limit.

Recall that for a bounded self-adjoint operator $J$ on a Hilbert space, the spectral measure of $J$ at a unit vector $v$ is the measure whose $k$th moment is $\langle v,J^kv \rangle$. We will omit mentioning the vector $v$, and by default take it to be the first coordinate vector. The spectral measure of a matrix $M$ is then the weighted sum of delta masses at the eigenvalues with weights given by the squared first entries of the corresponding normalized eigenvectors.

Let $J$ denote the Jacobi operator, or semi-infinite symmetric tridiagonal matrix, whose entries $a_i$, $b_i>0$ solve the optimization problem \eqref{e:generalH} with $\alpha_i=1$, $n=\infty$, and left boundary conditions $0$.

More precisely, its entries are the unique bounded sequence so that for any finite subinterval of indices $I$, the variables with indices of $I$ solve the (now finite) optimization given the rest as boundary conditions.
\begin{lemma} The operator $J$ defined above exists, is unique, and satisfies  that
for some $c>0$, as $k\to\infty$ we have
\begin{equation}\label{e:expclose}
(J_{k,k},J_{k,k+1}) = (a^\dagger(0),b^\dagger(0)) + O(e^{-ck}).
\end{equation}
\end{lemma}
\begin{proof}
The solutions for the optimization problem with zero boundary conditions on the interval $[1,n]$ have to converge to a limit  $J$ by Proposition \ref{prop:minbound}, which says that the values of two solutions with different boundary conditions are close away from the boundary.  The same Proposition implies that the limit $J$ will solve the optimization problem restricted to subintervals, as in the definition.

Uniqueness also follows from Proposition \ref{prop:minbound}: $J$ and any different solution has to be arbitrarily close in any subinterval.

Finally, Proposition \ref{prop:minbound} also implies \eqref{e:expclose}. Indeed, the vector with entries $a^\dagger(0),b^\dagger(0)$ minimizes the same Hamiltonian as the on- and off-diagonals of $J$ with different boundary conditions.
\end{proof}

The first step toward Theorem \ref{t:beginning} shows that the leading minors of $T_n$ are well approximated by those of $J$.  For clarity, throughout the rest of this section we use the notation $A[\ell,k]$ for the minor of a matrix $A$ drawn from the rows and columns with indices in the set $[\ell, k]$.

\begin{proposition}\label{p:lln}
For the matrix $T^\circ_n$ of global minimizers for $\beta\in [1,\infty]$ and for every $c\log n <m<c'n$ we have
\begin{equation}\label{e:lln1}
\|T^\circ_n[1,m]-J[1,m]\|\le c \frac{\max(m,\log^2 n)}{n}
\end{equation}
Moreover, for all $\beta\in (0,\infty)$ with probability at least $1-ce^{-c\log^2 n } $ we have
\begin{equation}\label{e:lln2}
\|T_n[1,m]-J[1,m]\|\le c \frac{\max(m,\log^2 n)}{n}.
\end{equation}
\end{proposition}
It follows immediately that $J$ encodes the limiting empirical eigenvalue distribution, the analogue of Wigner's semicircle law.
\begin{corollary} \label{c:lln}The spectral measure $\mu$ of $J$ equals
\begin{enumerate}[(i)]
 \item \label{n:wsc} the limit of the eigenvalue distribution of $T_n$ for every $\beta$ and
\item \label{n:equilibrium} the limit of the empirical distribution of the Fekete points, the $n$-point minimizers of the density  $ c  e^{-\sum  V(\lambda_i)}\prod_{i<j} |\lambda_i-\lambda_j|$.
\end{enumerate}
\end{corollary}

The second characterization (\ref{n:equilibrium}) identifies $\mu_V$ as the classical {\bf equilibrium measure} of potential theory associated with the potential $V$. In particular, $\mu_V$ and $J$ are non-universal. The fact that (i) equals $\mu_V$ regardless of $\beta$ is classical; we include it to clarify how it fits into our framework. The proof a simple consequence of the steps we need to prove our main result.

\begin{proof}[Proof of Corollary \ref{c:lln}]
For claim (ii), note that the $\beta=\infty$ energy associated with the tridiagonal matrix $T_n(a,b)$ can be expressed in terms of $q_i,\lambda_i$ using the expression \eqref{spectralmapid}:
\ben \notag
 H(a,b)=\tr(V(T_n(a,b)))+\summ_{k=1}^{n-1}(1-k/n)\log(b_k)
 =\log \left(\prod_{i=1}^n q_i e^{-V(\lambda_i)} \ \prod_{i<j}|\lambda_i-\lambda_j|\right)+c
\een
so $H(a,b)$ is minimal exactly when the $\lambda_i$ are the  Fekete points, and all the $q_i$ are equal. In particular, for $\beta=\infty$, the spectral measure of the tridiagonal matrix  $T^\circ_n=T_n(a^\circ,b^\circ)$ built from the minimizers of $H(a,b)$ is the same as the empirical distribution of the Fekete points.

By the Proposition, the $k$-th moment of this measure is given by the $\beta=\infty$ case  $$((T^\circ_n)^k)_{1,1}=((T^\circ_n[1,k])^k)_{1,1}\to ((J[1,k])^k)_{1,1}=(J^k)_{1,1},$$
and so these measures converge weakly to the spectral measure of $J$, showing \eqref{n:equilibrium}.

The same argument, using $T_n$ in the finite $\beta$ case, shows that the spectral measures converge to that of $J_k$. Finally, recall Proposition \ref{p:matrix model} that for the $\beta$-ensemble, the spectral measure is just a reweighted version of the empirical distributions, where the weights are Dirichlet$(\beta/2,\ldots, \beta/2)$, independent from the eigenvalues. If we realize the Dirichlet distribution as independent Gamma variables normalized by their sum, \eqref{n:wsc} follows easily from the law of large numbers.
\end{proof}

\begin{proof}[Proof of Proposition \ref{p:lln}]
We first handle the $\beta \in[1,\infty]$ case.

Let $T'_{n,m}$ be the $m\times m$ matrix built from the minimizer of the main Hamiltonian $H=H_n$ of \eqref{Hdef}
with boundary conditions given by the corresponding entries of $J$ on $m-\deg V/2,\ldots, m$ on the right, and zero on the left.
The entries of $T'_{n,m}$ and  $J[1,m]$ minimize different Hamiltonians with the same boundary conditions. Proposition \ref{p:twoH} gives the bound
\begin{equation}\label{e:JT'}
\|J[1,m]-T'_{n,m}\| \le c m/n,
\end{equation}
in a slightly altered form: for the maximum of entries of the difference matrix. The operator norm bound follows from this and close and the Greshgorin circle theorem (the matrix is diagonally dominant).

To compare $T'_{n,m}$ to $T^\circ_n$,
Corollary \ref{c:conditional-local} gives that for $j=m-\deg V/2, \ldots ,m$ the entries $a^\circ ,b^\circ $ of $T^\circ$ satisfy
$$\|(a^\circ_j,b^\circ_j)-(a^\dagger_j,b^\dagger_j)\|\le c \frac{\max(m,\log^2 n)}n,$$
by analyticity of $a^\dagger, b^\dagger$ we have
$$\|(a^\dagger_j,b^\dagger_j)-(a^\dagger(0),b^\dagger(0))\|\le c\frac mn.$$
Now \eqref{e:expclose} gives that for $m\ge c \log n$ we have
$$\|(a^\circ_j,b^\circ_j)-(a^\dagger(0),b^\dagger(0))\|\le c /n.$$
So the boundary conditions for $T'_{n,m}$ and $T^\circ_n$ are indeed very close, and they minimize the same Hamiltonian. This implies that the minimizers are also close by Proposition \ref{prop:close}:
$$
\|T'_{n,m}-T^\circ_n[1,m]\|\le c \frac mn
$$
which, together with \eqref{e:JT'} shows the first claim \eqref{e:lln1}.

Now assume $\beta\in[1,\infty)$. Proposition \ref{p:variation} gives the entries $A,B$ of the random tridiagonal matrix $T_n$ satisfy
$$
\P(|(A_k,B_k)-(a^\dagger_k,b^\dagger_k)|<cn^{-1/2+\delta}\mbox{ for all }  m-\deg V/2\le k\le m )\ge 1- e^{-cn^\delta}
$$
now given these entries, the $A_k,B_k$ for $k<m-\deg V/2$ are picked from the conditional Hamiltonian. These conditional minimizers are now $n^{-1/2+\delta}$-close to $T^\circ$ by Proposition \ref{prop:close},  since the boundary conditions are so close (with $m'=\max(m,\log^2 n)$). The Gaussian Lemma \ref{l:GaussianLem} shows that the deviations from the minimizer are bounded above by constant times the norm-squared of an i.i.d. Gaussian vector of dimension $m$ and variance $c/n$. So it is exponentially (in $m$) unlikely to be more than $cm/n$. In summary, we get that with probability at least $1-e^{c\log^2n}$, we get the desired conclusion
$$
\|T_n[1,m]-T^\circ_n[1,m]\|\le c \frac{m'}n \,.
$$
The $\beta<1$ case can be treated by conditioning on the $b$ terms with positive log coefficient the same way as in Proposition \ref{prop:cheapbound}.
\end{proof}


Next, denote by $j_{0,0}$ the first zero of the Bessel
function $J_0$ of the first kind. It will be crucial for the proof that we have the strict inequality
$j_{0,0}\sim 2.40482
> \pi/2$.
Let  also $\lambdamax(A)$ denote the largest eigenvalue of a matrix $A$.
The point of Proposition \ref{p:lln} is that top corner of $T_n$ looks more and more like $J_k$, and next we show that
 $J_k$ has no large eigenvalues.

\begin{proposition} \label{p:quadraturetop}
We have that
$
\lambda_{max} (J_k) \le a^\dagger(0)+b^\dagger(0)(2-(j_{00}/k)^2)  + o(k^{-2}).
$
\end{proposition}
For the proof of this proposition, we recall a few facts from both orthogonal polynomial theory and potential theory.
\begin{fact} \label{f:potential} The following hold.

\begin{itemize}
\item
For any probability measure $\mu$ of bounded support there exists a unique Jacobi operator $J$  with spectral measure $\mu$. If the support of $\mu$ consists of $m$ points, then $J$ reduces to an $m\times m$ Jacobi matrix.
\item
The top $k\times k$ minor $J_k$ of $J$ is uniquely determined by the first $2k-1$ moments of $\mu$.
\item For $k\le m$ the spectral measure $\mu_k$ of $J_k$ satisfies
\begin{equation}\label{e:essup}
\essup (\mu_k) = \lambda_{max}(J_k) \le \essup (\mu).
\end{equation}
\item For the case when $\mu$ is the uniform measure on $[-1,1]$, then $\mu_k$ is supported
on the zeros of the $k$th Legendre polynomial. The Bessel asymptotics for the Legendre polynomials imply that
\begin{equation}\label{e:uniform}
\essup (\mu_k)=1-(j_{00}/k)^2/2  + o(k^{-2}).
\end{equation}
\item When $V$ is convex, the measure $\mu=\mu_V$ has concave density on its support $[-2b^\dagger(0)+a^\dagger(0),2b^\dagger(0)+a^\dagger(0)]$.
\end{itemize}
\end{fact}

The first item  is typically referred to as Favard's Theorem \cite{Akhiezer}. The second is implied by Heine's formulas, see again  Section 3.1 of \cite{Deiftbook}.
The estimate on the largest  zero of the $k$th Legendre polynomial may be found in  Chapter 6 of \cite{szego}.

Assembling the above items leads to the next lemma, from which Proposition \ref{p:quadraturetop} quickly follows.

\begin{lemma}\label{l:essup} Let $\mu$ be a probability measure with support $[-2,2]$ and having concave density. Then for any $m\ge 1$, there is a measure that has the same first $2m-1$ moments as $\mu$ and is supported on  a subset of $[-2+m^{-2}\eta_m,2-m^{-2}\eta_m]$, where
$$
 \eta_m \tends \eta = j_{0,0}^2.
$$
\end{lemma}
\begin{proof}
For a sufficiently
small $\eps$, we may write $\mu$ as a convex combination of
$\nu_t$, $t\le \eps$ and $\nu$, where, $\nu_t$ is the uniform
measure on $[-2+\eps,2-t]$ and $\nu$ is supported on
$[-2,2-\eps]$.

For each $\nu_t$, by \eqref{e:uniform}, we may find a measure supported on
$
[-2,2-(2-\eps)m^{-2}\gam_m]$ and having the same first $2m-1$
moments as $\nu_t$.

Thus we may find a measure supported on
$$
[-2,2-\theta_{m,\eps}],\qquad \theta_{m,\eps}=
(2-\eps)m^{-2}\gam_m \wedge \eps
$$
and having the same first $2m-1$ moments as $\mu$. Choose
$\eps_m=1/m$ and set $\eta_m=(2-1/m)\gam_m$ to get the desired claim.
\end{proof}

\begin{proof}[Proof of Proposition \ref{p:quadraturetop}]
By shifting and scaling we may assume $a^{\dagger}(0)=0$, $b^{\dagger}(0)=1$.
Lemma \ref{l:essup} shows that there is a measure $\mu'$ that has the same first $2k-1$ moments as $\mu$ and satisfies $\essup(\mu')\le 2-(j_{00}/k)^2 +o(k^{-2})$. By Fact \ref{f:potential}, the matrix $J'$ corresponding to $\mu'$ has the same $k\times k$ top minor as $J$, and so by \eqref{e:essup} $\lambdamax(J_k) \le \essup(\mu')$, as required.
\end{proof}

Finally, the following theorem will be used to show that for the sake of edge asymptotics we can safely ignore the first $m\times m$ block of the matrix $T_n$ as long as $m$ is not too large. For this theorem, let $T_{|[a,b]}$ denote the matrix $T$ where all entries except the ones with both indices in $[a,b]$ are set to zero.
\begin{theorem}\label{t:beginning} There exists $c,c_0,\delta,\kappa>0$ so that with  probability at least $1-e^{-n^{\delta}}$ for every $m\in [c\log n,c_0n^{1/3}]$, and large enough $n$ the following holds in the positive definite order:
$$
T \le (b^\dagger (0)-\frac{\kappa}{m^2})e_{mm} +T_{|[m,n]}+( 2b^\dagger(0)+a^\dagger(0)-\frac{\kappa}{m^2})I_{|[1,m-1]}.
$$
\end{theorem}
We will need the following simple lemma. Let $\lambdamax$ denote the top eigenvalue, and let $e_{ij}$ denote the elementary matrix with zeros everywhere except for a one at entry ${i,j}$; its dimension will be clear from the context.
\begin{lemma}\label{l:oplus}
Let $T$ be an $n\times n$ tridiagonal matrix with positive off-diagonals, let $m\in \{1,\ldots, n-1\}$, and  let $q=T_{m,m+1}$. We have
$$
\lambdamax(T)=\max_{r>0} \min\Big(\lambdamax(T[1,m]+q \,r \; e_{mm}),\lambdamax(T[m+1,n]+\frac{q}{r}\; e_{11})\Big).
$$
\end{lemma}

\begin{proof}
The two $\lambdamax$ expressions on the right are nondecreasing (respectively, nonincreasing) as functions of  $r$, so it suffices to show that for some $r>0$ we have
$$
\lambdamax(T)=\lambdamax(T[1,m]+q \,r \; e_{mm})=\lambdamax(T[m+1,n]+\frac{q}{r}\; e_{11}.
$$
By adding a sufficiently large constant times identity to $T$ we may assume without loss of generality that it has nonnegative entries.
By the Perron-Frobenius theorem, there exists an eigenvector $\varphi$ of $T$ with nonnegative entries for the positive eigenvalue $\lambdamax(T)$. If there were neighboring coordinates $i,j$ so that $\varphi_i=0$ and $\varphi_j\not=0$, then the eigenvalue equation at coordinate $i$ would fail. Thus $\varphi$ is strictly positive.

Now with $r=\varphi_{n+1}/\varphi_{n}$, $\varphi$ restricted to the first $m$ coordinates is an eigenvector of $T[1,m]+q r\; e_{mm}$ with eigenvalue $\lambdamax(T)$.  Moreover, it must be a top eigenvector: by the Perron-Frobenius theorem, there is a nonnegative top eigenvector which cannot be orthogonal to the positive $\varphi$.
Similarly,  $\varphi$ restricted to the last $n-m$ coordinates is a top eigenvector of $T[m+1,n]+\frac qr \,e_{11}$ with eigenvalue $\lambdamax(T)$, as required.
\end{proof}

\begin{proof}[Proof of Theorem \ref{t:beginning}]
By shifting and scaling we may assume that $a^{\dagger}(0)=0$, $b^{\dagger}(0)=1$. We decompose $T$ as follows:
\begin{eqnarray*}
T&=&\left(T_{|[1,m]}+(1-T_{mm})e_{mm}-(2-\frac{\kappa}{m^2})I_{|[1,m]}\right)
\\&&+\;\;
  \left(T_{|[m,n]}+(2-\frac{\kappa}{m^2})I_{|[1,m-1]}
+(1-\frac{\kappa}{m^2})e_{mm}\right)
\end{eqnarray*}
It suffices to show that the nontrivial part
$
T[1,m]+(1-T_{mm})e_{mm}-(2-\frac{\kappa}{m^2})I_m
$
of the first matrix
is negative definite with high probability. In light of Proposition \ref{p:quadraturetop}, and \eqref{e:expclose},
we have
$$
|T_{mm}|,  \| T[1,m]-J[1,m]\|<
c\max(m,\log^2 n)/n<cc_0/m^2
$$
and so
(by decreasing $\kappa$ a bit and choosing $c_0$ small enough) it suffices to show that
\begin{equation}\label{e:nonnegative}
\lambdamax(J[1,m]+e_{mm})<2-\frac{\kappa}{m^2}.
\end{equation}
Now Proposition \ref{p:quadraturetop} gives us control of $\lambdamax(J[1,m])$, but here we clearly need a little bit more. For this, we let $p>1$ so that $pm$ is an integer, and study the matrix $J[1,pm]$.

By \eqref{e:expclose} the matrix $J[m+1,pm]$ is exponentially close in norm to the $\ell\times \ell$ matrix $D_\ell$ with 1-s on the first off-diagonals and zeros elsewhere (and $\ell=pm-m+1$). It is easy to check that
$$
\lambdamax (D_\ell+e_{11})=\lambdamax (D_{2\ell})=2\cos(\frac{\pi}{2\ell+1})=2-\frac{\pi^2}{4\ell^2}+O(\ell^{-3}).
$$
Since $q=J_{m,m+1}$ is exponentially close to $1$ by \eqref{e:expclose}, we have
$$\lambdamax (J[m+1,pm]+q^2e_{11}) = 2-\frac{\pi^2}{4(p-1)^2m^2}+O(m^{-3}).$$
If $p$ is large enough so that  $\pi/(2j_{00})>1-1/p  $, then for large  $m$ this is bounded below by
$$2-\frac{j_{00}^2}{p^2m^2}+o(m^{-2})\ge \lambdamax (J[1,pm]),$$
the last inequality coming from Proposition \ref{p:quadraturetop}.
Lemma \ref{l:oplus} applied to $J[1,pm]$ with $r=1/q$ now gives
$$
\lambdamax (J[1,m]+ e_{mm})\le \lambdamax (J[1,pm]) =2-\frac{j_{00}^2}{p^2m^2} + o(m^{-2})
$$
showing \eqref{e:nonnegative} and thus the Theorem.  In fact, carefully following the constants shows that we may use $\kappa=1/2$.
\end{proof}

\section{Mean and variance of the limiting potential}
\label{s:meanandvariance}

The final ingredient needed to establish our main result is the computation
of mean and variance of the limiting (Brownian) potential of the operator identified by Theorem \ref{weak}
(the limit of the processes defined in \eqref{summedpotentials}).

Proposition \ref{p:variation} shows that the field over the required range is close to a mixture of Gaussians, centered at certain conditional minimizers
and with variance given in terms of the inverse Hessian of the associated conditional Hamiltonian (evaluated at those conditional minimizers).  Further,
with the event $S$ in that statement holding with high probability, any of these conditional minimizers are close to the corresponding local minimizers, recall
\eqref{e:Sbound}.   Thus we expect the limiting mean and variance to be given only in terms of the local quantities.

Return to the given (by the Dyson beta ensemble) conditional Hamiltonian:
for a set of indices $I$,
\begin{equation}
  \label{e:HI}
H_I(a,b)=V_I(a,b)-\summ_{k\in I\cap\{1,\ldots, n-1\}}\left(1-\frac kn-\frac1{n\bet}\right)\log b_k,
\end{equation}
where again $V_I(a,b):=\tr(V(T_I(a,b))$ for the minor $T_I$ tied to the index set $I$.

For the mean, the following provides the basic ingredient. The condition \eqref{eq:lldiv1} below distills what is required out of Proposition \ref{p:variation}.
 It is a deterministic condition that implies the bound \eqref{e:Sbound} as a direct consequence of Corollary \ref{c:conditional-local}.

\begin{lemma}
\label{l:means}
Fix $\delta>0$. Let $I=\{\lfloor c\log n\rfloor ,\ldots ,\lfloor n^{1/2-\delta}\rfloor \}$ and
assume that for some $a_i,b_i$ we have
\begin{equation}\label{eq:lldiv1}
|a_i-a_i^{\dag}|+|b_i-b_i^{\dag}|\le n^{\delta-1/2} \qquad \mbox{for all } i\in \partial I,
\end{equation}
where $(a_i^\dag, b_i^\dag)$ are the local minimizers defined in \eqref{bdagger}.
In $I$, let $(a, b)$
be the minimizers of the conditional Hamiltonian $H_I$  with the above boundary conditions.  It holds that
\begin{equation}
\label{e:means}
    n^{1/3}  \sum_{k=c\log n}^{ \lfloor x n^{1/3} \rfloor} ( a_k - a^{\dagger}_0 )  \rightarrow \frac{1}{2}  (a^{\dagger})'(0) {x^2}, \quad
      n^{1/3}  \sum_{k=c\log n}^{ \lfloor x n^{1/3} \rfloor} ( b_k - b^{\dagger}_0 )  \rightarrow \frac{1}{2}  (b^{\dagger})'(0) {x^2}
\end{equation}
as $n \rightarrow \infty$.
\end{lemma}

\begin{proof} Consider the expression over the $a$-variables (the treatment for the $b$-variables is identical). Re-centering to arrive at
$$
    n^{1/3}  \sum_{k=c\log n}^{ \lfloor x n^{1/3} \rfloor} ( a_k - a_k^{\dagger})  +   n^{1/3}  \sum_{k=c\log n}^{ \lfloor x n^{1/3} \rfloor} ( a_k^{\dagger} - a^{\dagger}_0 ),
$$
the claimed evaluation comes from the second sum. By the analyticity of the local minimizer $x \mapsto a^{\dagger}(x)$ this sum may be replaced by
$n^{1/3} \sum_{k=1}^{ \lfloor x n^{1/3} \rfloor} ( a^{\dagger}(k/n) - a^{\dagger}(0) )$ which, again by analyticity, is asymptotic to $ (a^{\dagger}(0))' $ times
$n^{1/3}  \sum_{k=1}^{ \lfloor x n^{1/3} \rfloor} (k/n + O(k^2/n^2)) = x^2/2 +O(n^{-2/3})$.

As for the first sum, Corollary \ref{c:conditional-local}  implies that, given that \eqref{eq:lldiv1} is in place, we have the bound
$$
|a_k -a_k^\dagger| \le c \max\left( n^{\delta-1/2} \exp(-c_1\dist(k, \partial I)), \frac{(\log n)^2}{n}\right).
$$
It follows that the sum of $n^{1/3}|a_k -a_k^\dagger|$ up to $xn^{1/3}$ is
bounded above by $cxn^{1/3}n^{\delta-1/2}\log n\to 0$, as required.
\end{proof}

For the variance we record:

\begin{lemma}
\label{l:covariance}
Keep the setup of the previous lemma: given $\delta >0$ and  $I=\{\lfloor c\log n\rfloor,$ $\ldots  ,\lfloor n^{1/2-\delta}\rfloor \}$
assume that for some $a_i,b_i$ we have
\begin{equation*}
|a_i-a_i^{\dag}|+|b_i-b_i^{\dag}|\le n^{\delta-1/2} \qquad \mbox{for all } i\in \partial I.
\end{equation*}
Here again $(a_i^\dag, b_i^\dag)$ are the local minimizers, and in
$I$, let $(a, b)$ are the minimizers of the conditional Hamiltonian $H_I$  with the above boundary conditions. Fix $0\le x<y$, and let
and let $\xi^J_a$ be the indicator of an subinterval $\lfloor\max(c\log n, xn^{1/3})\rfloor, \ldots,\lfloor yn^{1/3}\rfloor$ of $I$, of
the $a$ variables. Define $\xi^J_b$ analogously.
Let $\Hb(a,b)=\Hess_{H_I}(a,b)$. Then as $n\to\infty$ we have
\begin{equation}
\label{eq:covariance}
n^{-1/3}\langle\xi^J_{i}, \Hb^{-1} \xi^J_{j} \rangle  \to (y-x)\Sigma(0)_{ij}
\end{equation}
where $i,j\in\{a,b\}$ and the covariance matrix $\Sigma(x)$ is defined in \eqref{e:sigma}.
\end{lemma}

\begin{proof}
By polarization and linearity, it suffices to show that
for any normalized eigenvector $(\alpha_a,\alpha_b)$ of $\Sigma(0)$ with
eigenvalue $\lambda$, and
$w=w_n=\alpha_a \xi^J_a+\alpha_b \xi^J_b$ we have
$$
n^{-1/3}\langle w, \Hb^{-1} w \rangle  \to (y-x)\lambda.
$$
Consider the Hessian of the local Hamiltonian \eqref{localH} evaluated at its minimizer.
This is invariant under the rotation of the indices, so the space spanned by the indicator $\xi^a$ of all variables $a$ and the indicator
$\xi^b$ of all $b$ variables form an invariant subspace. Then by the discussion before Proposition \ref{p:localcovariance} it holds that
$v=\alpha_a \xi_a+ \alpha_b \xi_b$ is an eigenvector with eigenvalue $\lambda^{-1}$.

Note that the matrix $\Cb$ is a block circulant
with $2\times 2$ blocks. It is also a band matrix -- any blocks entry whose two block
indices are further than $\deg V/2$ apart mod $|I|$ are zero. This implies
that if two vectors $v$ agree on a long interval (of blocks), then this
will hold after applying $\mathbb H$, apart from a short section of $\deg_V$ blocks
that could have changed. Note that the vector $v$ agrees
with $v$ on a long stretch and then with $0$ on another long stretch.
Thus we  have
$$
\|\Cb w -\lambda^{-1} w\| \le c
$$
the constant comes from the boundary error terms, which are bounded and
are of a bounded number. Moreover, the same consideration shows that if $\Cb'$ is
$\Cb$ with the periodic boundary removed (i.e.\ all block entries with indices farther
apart than $\deg V/2$ changed to zero), we have
$$
\|\Cb' w -\lambda^{-1} w\| \le c
$$
Note that $\Cb'$ is Hessian for the non-periodic Hamiltonian \eqref{e:generalH}, with
all coefficients of the log terms there given by $1$. The Hessian is
for the variables with indices in $I$,
evaluated when all variables (including the boundary conditions with indices
in $\partial I$) are set to $a^\dagger(0),b^\dagger(0)$, respectively. In particular,
$\Cb'$ is bounded below by a constant times the identity, which implies
$$
\|\lambda w -\Cb^{-1}w\| \le \lambda \|\Cb'^{-1}\|  \|\Cb' w -\lambda^{-1} w\| \le c.
$$
We can now use the assumption on the boundaries along with Corollary \ref{c:conditional-local}
  to compare the entries of $\Cb'$ and $\Hb$.
These are given by polynomials of $a_i, b_i$, plus a
constant times $b_i^{-2}$, for $\Hb$, and exactly the same
expressions with  all $a_i,b_i$ replaced by $a^\dagger(0),b^\dagger(0)$ for
$\Cb'$. By the proposition the $a_i$ are $cn^{\delta -1/2}$-close to $a^\dagger_i$,
and $a^\dagger_i=a^\dagger(i/n+o(1/n))$  are at most $cn^{\delta-1/2}$-close to
$a(0)$. This is because the function $a$ is analytic and so Lipschitz in a neighborhood of $0$ (\ref{p:analytic}),
and we are on a stretch where $i/n \le c n^{1/2-\delta}$.
The same considerations hold for the $b$ variables.

Since $\Hb$ and $\Cb'$ are band matrices, the Greshgorin circle theorem
implies that for the operator norm $\|\Hb-\Cb'\|\le cn^{\delta-1/2}$, and so
$$
\|\Hb^{-1}-\Cb'^{-1}\|=\|H^{-1}(\Cb'-\Hb)\Cb'^{-1}\| \le \|H^{-1}\|\|\Cb'-\Hb\|\|\Cb'^{-1}\| \le cn^{\delta-1/2}.
$$
Thus we get the desired conclusion
\[
\langle w,\Hb^{-1} w\rangle =\lambda^{-1}\|w\|^2 + O(n^{1/3}n^{\delta-1/2})+O(1).
\qedhere \]
\end{proof}

\section{Putting it all together}
\label{s:mainresult}

Everything is now in place to prove our main result, which state more completely as follows.

\begin{corollary}
\label{mainresult}
Let $ \gamma = (b^{\dagger}(0))^{-1/3} \tau^{-2/3}$ for $\tau =  - (a^{\dagger})'(0)  - 2 (b^{\dagger})'(0) $, let $\vartheta= b^\dagger(0)/\tau$, and let 
$\mathcal E= a^{\dagger})'(0)  + 2 (b^{\dagger})'(0) $. 
There exists a coupling of the random matrices $T_n$ on the same probability space so that
a.s.\, we have
$$\gamma n^{2/3}(\mathcal E-T_n)\to \SAO$$
in the norm-resolvent sense: for every $k$ the bottom $k$th eigenvalue converges the and corresponding eigenvector converges in norm.
Here $\mathcal E-T_n$ acts on $\mathbb R^n\subset L^2(\mathbb R_+)$ with coordinate vectors $e_j=(\vartheta n)^{1/6}\one_{[j-1,j](\vartheta n)^{-1/3}}$. 
\end{corollary}

The next two subsections check the conditions of Theorem \ref{weak}, in particular how part (i) of that theorem applies to the (centered and scaled) truncated  matrix
$H_n = \gamma  n^{2/3} \, ( \mathcal{E} I - T_{[c \log n, n]} )$.  After this, the results of Section \ref{s:firststretch} are incorporated to bypass the truncation and prove the corollary.

\subsection{Tightness and Brownian convergence}

The following establishes Assumption 1 of Theorem \ref{weak}, in particular the tightness and convergence of the processes
$x \mapsto y_{n,1}(x), y_{n,2}(x)$ defined in \eqref{summedpotentials}.  Giving the limit of the (integrated) potential, this result
already identifies the Stochastic Airy Operator.

Here we can assume that the conclusions of Proposition
\ref{p:variation} are in place. Said an other way we can assume we are on the event $S$ defined there.

\begin{proposition}
\label{thm:mainconv}
Consider the measures on paths
$$
X_n(x)= n^{1/3} \sum_{k= c \log n}^{\lfloor n^{1/3} x \rfloor} (b_0^{\dagger} - B_k  ),
$$
$$
Y_n(x)=n^{1/3} \sum_{k=c \log n}^{\lfloor n^{1/3} x \rfloor} (a_0^{\dagger} - A_k).
$$
For $c$ chosen large enough,
$X_n$ and $Y_n$ from a tight family with respect to the
uniform convergence on compact sets. Moreover, we have the
following convergence in distribution with respect to the
uniform-compact topology:
\begin{equation}
\label{eq:firstconv}
X_n+ 2 Y_n \Rightarrow  \sigma W_x   + \frac{\tau}{2} x^2.
\end{equation}
Here $x \mapsto W_x$ is a standard Brownian motion and
\begin{equation}
\label{eq:stats}
 \tau = - \mathcal{E}'(0) = - \Bigl(  (a^{\dagger})'(0)  + 2 (b^{\dagger})'(0) \Bigr), \quad \sigma^2 = \frac{4}{\beta} b^{\dagger}(0) \tau.
\end{equation}
Recall from \eqref{eq:edgemonotone} that $\tau > 0$.
\end{proposition}

\begin{proof}
We first switch to variables picked from the distribution $\nu_I$
introduced before Proposition \ref{p:variation}. We will still denote these random
variables
$A_k,B_k$.  The $c$ in the statement should simply be large enough that we can assume
the outcome of Proposition \ref{p:variation}.

To prove tightness, note that it suffices to do this for the
components of the mixture of $\nu_q$'s $\nu_I$, since mixtures from tight families
are also tight.
Thus let
$$Z_n(x)=
n^{1/3} \sum_{k=c \log n}^{\lfloor n^{1/3} x \rfloor} (B_k-b_k^\scond)
$$
be the partial sums as above from the sample from
the normal distribution corresponding to the variables $B_k$.

By Kallenberg \cite{Kallenberg}, Corollary 16.9, in order to establish tightness for the
sequence $Z_n$
it suffices to show that
the Kolmogorov-Chentsov criterion holds  $Z_n$ and uniformly for
all $n$ and boundary conditions $q$. In our case, this comes down to showing that
\begin{equation}\label{e:preKC}
\ev |Z_n(x)-Z_n(y)|^2 \le c |x-y| \mbox{ for all $x,y$ such that }  |x-y|<1.
\end{equation}
Since the increments are normally distributed, this automatically
implies a 4th moment bound that works as an input for the
Kolmogorov-Chentsov criterion, Lemma \ref{l:KC}.

Indeed, Lemma \ref{l:UniformConvexity} implies that the Hessian $\Hess \ge c_uI$, whence $\Hess^{-1}\le c_u^{-1}I$ uniformly over $q$ and $n$. Let $v$ be the indicator of the
coordinates of $\Hess$ corresponding to $B_{k}, \ldots B_{\ell}$. Then
we have that
$$
\ev [ (B_{k} - b_k^\scond) + \ldots + (B_{\ell-1} - b_{\ell-1}^\scond ) ] ^2 = \langle v, (n\Hess)^{-1}v \rangle \le \frac{c}{n}\langle v,v\rangle
=\frac{c(\ell-k)}{n}
$$
which, after appropriate scaling, gives precisely \eqref{e:preKC}.

We also need to prove tightness of the integrated drifts,
$$
n^{1/3} \sum_{k=c \log n}^{\lfloor n^{1/3}x \rfloor} ( b_0^{\dagger} -b_k^\scond).
$$
However, this along with the convergence to the desired value of $ - \frac{1}{2} (b^{\dagger})'(0) \, x^2 $ is the content of Lemma
\ref{l:means}.  Note again what is required from the condition \eqref{eq:lldiv1} in that lemma is just  the final statement of
Proposition \ref{p:variation}.  In the same way Lemma \ref{l:covariance} shows that $\Var [Z_n(x)]$ converges.

Since each $Z_n$ is a Gaussian process, the convergence of the means and variances (plus tightness) is enough to yield unique limit in the
 uniform-on-compacts topology.  That the limit process is the appropriately shifted and scaled Brownian motion (to match the given mean and variance)
is implicit in the covariance structure described in Lemma \ref{l:covariance}.

This establishes tightness and convergence for $Z_n$. Again, that for $Y_n$ follows from Proposition \ref{p:variation} which shows that the
processes are close in total variation.
The same procedure gives the desired conclusions for the $X_n$ process. The integrated drifts in that case converging to $  - \frac{1}{2} (a^{\dagger})'(0) \, x^2 $,
thus giving the formula for $\tau$.

To verify the formula for $\sigma^2$, return to  the formula given in Proposition \ref{p:localcovariance} and compute
\[
 \beta \sigma^2
 := \Sigma(0)_{11} +4 \Sigma(0)_{22} +4\Sigma(0)_{12}
=  -4 b^{\dagger}(0) \Bigl( (a^{\dagger})'(0) + 2 (b^{\dagger})'(0) \Bigr).
\qedhere
\]
\end{proof}

\subsection{Oscillation and growth bounds}

We establish the non-immediate conditions of Assumption 2, the bounds \eqref{bounds1} and \eqref{bounds3}.
Recall from Section \ref{s:outline} that the definition of growth and oscillation terms differs depending on the range of
indices.

\subsubsection*{On the interval [$c \log n, (1-\epsilon) n$]}

The $\epsilon$ is as defined in Proposition \ref{p:variation}, and again we can assume that we are on the event $S$ on which
the various bounds described in that proposition all hold.  On this stretch, the oscillation (or noise) terms are defined, up to constants,
by $
w_{n,1}(x)=   n^{1/3} \sum_{k=  \lfloor c \log n   \rfloor}^{ \lfloor x n^{1/3} \rfloor } (a^{\dagger}_k - A_k )$, and similarly for $w_{n,2}$ in the $B$-variables.
The growth terms, $\eta_{n,i}(x)$,  then read, again up to constants, $n^{2/3} (a_0^{\dagger} - a^{\dagger}_{ \lfloor n^{1/3} x \rfloor})$ and
$ n^{2/3} (b_0^{\dagger} - b^{\dagger}_{ \lfloor n^{1/3} x \rfloor})$. Compare \eqref{noise} and \eqref{driftderivatives}.

From the established differentiability of $x \mapsto a^{\dagger}(x), b^{\dagger}(x)$ we see that $ x/c \le  \eta_{n,i}(x) \le c x$ with a (deterministic) constant,
so that \eqref{bounds1} is satisfied with $\bar{\eta}(x) = x$. The following then shows \eqref{bounds3} is satisfied with $\zeta(x) = x^{\delta}$ for some $\delta > 0$ small.

\newcommand{\osc}{{\operatorname{osc}}}
\begin{proposition}[Oscillation bounds]
\label{p:osc}
Let $k_0=c\log n$, and for $0\le x \le (1-\eps)n$ and some sequence $\alpha_i$ denote
$$
\osc_{n,x}(\alpha) =\max_{0\le\ell<n^{1/3}} \left(n^{1/3}\sum_{i=k_0+xn^{1/3}}^{k_0+xn^{1/3}+\ell}\alpha_i\right)^2 \le C_n(1+x^{1-\delta})
$$
then for some $\delta>0$ and a tight sequence of random constants $C_n$ we have
for all $n$ and $0\le x \le (1-\eps)n$
$$
\osc_{n,x}(A_n-a^\dagger_n) + \osc_{n,x}(B_n-b^\dagger_n)  \le C_n(1+x^{1-\delta}).
$$
\end{proposition}

\begin{proof}
We will show the claim for $\osc_{n,x}(A_n-a^\dagger_n)$, the proof for the $a$
variables is identical.

Partition the interval $k_0, \ldots, (1-\eps)n$ into
with subintervals $I_k$ of size $n^{1/2-\delta}$.
Assume
that the indices $k_0+xn^{1/3}$ and $k_0+(x+1)n^{1/3}$ fall in the same
interval $I_k$ (otherwise we can split the sup and bound it in two
steps).

Let $$X_\ell=X_{n,x,\ell}=n^{1/3}\sum_{i=k_0+xn^{1/3}}^{k_0+xn^{1/3}+\ell} A_{n,i}-A^\sharp_{n,i}$$
where $A^\sharp$ is the (random) minimizer of the conditional Hamiltonian
of the variables in $I_k$, given the values on $\partial I_k$. Let
$S_{n,k}$ denote the event $S$ in Proposition \ref{p:variation}
applied to $I_k$.
Let $S_n=\bigcap_{k} S_{n,k}$. Note that by the Proposition $\P S_n\to 1$. Let
$
R'_{n,x}=\osc_{n,x}(A_{n}-A^\sharp_{n})
$ and
$
R''_{n,x}=\osc_{n,x}(a^\dagger_{n}-A^\sharp_{n})
$.
It suffices to show that
$$
\left(\max_{0\le x \le (1-\eps)n^{2/3}} R''_{n,x}\right)\one_{S_n} \to 0,
$$
in probability, and that we have
\begin{equation}\label{Ax}
\sum_{x=0}^{(1-\eps)n^{2/3}} \frac{\ev( R_{n_x}'^2;S_n)}{(1+x^{1-\delta})^2} \le c.
\end{equation}
For the first, note that by Proposition \ref{p:variation} on $S_n$ we have
$$
|A_j^\sharp - a_j^\dagger| \le c \max\left( n^{\delta-1/2} \exp(-c_1\dist(j, \partial J)), \frac{(\log n)^2}{n}\right).
$$
It follows that the sum of $n^{1/3}|A_j^\sharp  - a_j^\dagger|$ over any range of length $n^{1/3}$ is
bounded above by $c n^{1/3}n^{\delta-1/2}\log n\to 0$, as required.

By Proposition \ref{p:variation} on the event $S_n$ and
conditionally on $Q$ the random vector $(X_k,0\le k\le n^{1/3})$
is close to a Gaussian process with mean zero and $\ev (X_i-X_j)^2
\le c|i-j|/n^{1/3}$ in the sense of Radon-Nikodym derivatives
being close to 1. By Lemma \ref{l:KC} below we have that
$$
\ev[\,(R_{n,x}-R'_{n,x})^2;S_n\,|\, Q\,] \le 2 c (n^{-1/6})^4 (n^{1/3})^{4/2}=2c
$$
where the factor $2$ is to control the ratio of Radon-Nikodym
derivatives. Taking expectations \eqref{Ax} follows.
\end{proof}

\begin{lemma}\label{l:KC}
Let $(X(k),k=0, \ldots, n-1)$ be a centered Gaussian vector.
Assume that $\ev (X_i-X_j)^2\le \sigma^2|i-j|$ for all $i,j$. Then
for any $q>0$ we have
\[
\ev  \sup_{i} |X_i-X_0|^q \le c_{q}\, \sigma^{q}n^{q/2}.
\]
\end{lemma}

\begin{proof}
We prove this by showing a quantitative version of the
Kolmogorov-Chentsov criterion -- the proof is standard. We have
$$
\sup_{i,j} \frac{|X_i-X_j|}{|i-j|^r} \le  \frac{2^r}{2^r-1} \sup_{\ell,k} \frac{|X(k2^\ell)-X((k+1)2^\ell)|}{2^{\ell r}},
$$
as can be seen by repeated use of the triangle inequality. Indeed,
consider a shortest sequence $i_1=i,\ldots i_m=j$ so that
consecutive $i's$ differ by exactly one bit in their binary
expansion. Then by the triangle inequality we get
$$
|X_i-X_j| \le \sum_{\ell=0}^{\log_2 |i-j|} 2^{\ell r} s  = s \frac{2^r|i-j|^r-1}{2^r-1}
$$
where $s$ is the supremum on the right hand side. The claim
follows. Bounding the sup by the sum gives
$$
\left(\sup_{\ell,k} \frac{|X(k2^\ell)-X((k+1)2^\ell)|}{2^{\ell r}}\right)^q \le
\sum_{\ell=1}^{\lceil \log_2 n\rceil } \sum_{k=1}^{n/2^\ell}
\frac{(X(k2^\ell)-X((k+1)2^\ell))^q}{2^{\ell rq}}
$$

The expectation of the right hand side is bounded above by
$$
q!! \sum_{\ell=1}^{\lceil \log_2 n\rceil } n
2^{-\ell} \frac{(\sigma^2 2^\ell)^{q/2} }{2^{\ell rq}}=
q!!\sigma^qn \frac{(2^{\lceil \log_2 n\rceil
})^{q/2-rq-1}-1}{2^{q/2-rq-1}-1}\le c_{q,r}\,\sigma^q
(n^{q/2-rq}\vee 1)
$$
and we get
$$
\ev \left(\sup_{i,j} \frac{|X_i-X_j|}{|i-j|^r}\right)^q \le c'_{q,r}\, \sigma^{q}(n^{q/2-qr}\vee 1)
$$
for $r<1/2$ this gives, in particular
\[
\ev \sup_{i} |X_i-X_0|^q \le c_{q,r}\, \sigma^{q}n^{q/2}.
\qedhere\]
\end{proof}

\subsubsection*{On the interval $[(1-\epsilon) n, n]$}

Here we take the definitions in \eqref{e:endcondadjust}: $w_{n,i}(x) \equiv 0$ and
$ \eta_{n,1}(x) =  m_n^2 (a^{\dagger}_0 - A_{\lfloor x m_n \rfloor} ) / b^{\dagger}(0),  $
     $ \eta_{n,2}(x) =   m_n^2 (b^{\dagger}_0 - B_{\lfloor x m_n \rfloor} ) / b^{\dagger}(0).$

 By the proposition that follows we have a tight upper bounds on the $A$'s and $B$'s over this span of indices. This is sufficient for \eqref{bounds1} to hold as then
 $\eta_{n,i}(x) = O(n^{2/3})$ which is  $ O(x)$ throughout this span.  The  condition \eqref{bounds2} is  satisfied as the bound of $B_k$, $k \in [(1-\eps) n, n]$ can be made small:
 $B_k \le \eta$ (with high probability) where $\eta < b_0^{\dagger}/2$, for example.

\begin{proposition}[Last entries are small]
\label{c:finalstretch}
For every $\eta>0$ there exists $\delta>0$, with $\delta < 1$,  and $c_1,c_2,c_3>0$ so that we have
  \begin{equation}\notag
  \P\left(|A_k-a^\dagger(1)|,B_k\le \eta \mbox{ for all }k\ge \delta n\right) \ge 1-c_1\exp(-c_2n^{c_3}).
  \end{equation}
\end{proposition}

\begin{proof}
This is an immediate corollary
of Proposition \ref{prop:cheapbound} which shows there is  $\delta'>0$ so that
for all
$n$ and $s\ge 1-\delta'$ we have
\begin{align*}
  \label{e:cheap3again}
  \P\Bigl(|A_k-a^\dagger(1)|,B_k\le \eta \mbox{ for all } k\ge sn \,   &   \Big|\,|A_k-a^\dagger(1)|,B_k\le \delta' \mbox{ for all }k\in [sn,sn+\deg V] \Bigr)
   \nonumber \\
   & \ge 1-e^{-c_2n}.
\end{align*}
By Proposition \ref{p:analytic} we have $a^\dagger(x),b^\dagger(x)\to (a^\dagger(1),0)$ as $x\to 1$, and so
Proposition \ref{p:variation} shows that for some $\delta, s$ the conditioning event above holds with
the claimed probability.
\end{proof}

\subsection{Proof of the main result}

\begin{proof}[Proof of Corollary \ref{mainresult}]
Due to the interlacing inequalities, the lowest eigenvalues $\lambda_{n,1},\lambda_{n,2},\ldots$ of the matrix
\begin{equation}
\label{e:scaled}
     \gamma  n^{2/3} \, ( \mathcal{E} I - T_{n} )
\end{equation}
are bounded above by the lowest eigenvalues of any minor; we use the minor
 \begin{equation}
 \label{e:minor}
 H_n =   \gamma  n^{2/3} \, ( \mathcal{E} I - T_{[ c \log n, n]} ),
 \end{equation}
and denote its lowest eigenvalues by $\bar \lambda_{n,1},\bar \lambda_{n,2},\ldots$.

By Theorem \ref{t:beginning}, for $n$ large enough, the lowest eigenvalues of \eqref{e:scaled}  are also bounded below by $\min (\underline \lambda_{n,i},y_n)$ where
$y_n=cn^{2/3}/\log n$ and
$\underline \lambda_{n,i}$ are eigenvalues of a rank-1 perturbation of the same minor \eqref{e:minor} at its first coordinate. The perturbation adds the matrix
$$
 \gamma n^{2/3} \left(-b^{\dagger}(0) +\frac{c}{\log^2 n}\right) e_{11}  =    \left( - m_n^{2} + \frac{c'}{\log^2 n} m_n^2 \right) e_{11} := z_n e_{11}.
$$
Recall from \eqref{eq:mn} the definition $ (\gamma/n^{\dagger}(0)) n^{2/3} = m_n^2$.
Since $y_n\to\infty$, it suffices to show that the $\bar \lambda_{n,i}$ and the  $\underline \lambda_{n,i}$, the eigenvalues of $H_n$ and its perturbation
$H_n + z_n e_{11}$, converge to the same limit.

By Theorem \ref{weak} (ii), in particular the condition
\eqref{e:perturbation}, the perturbation here is subcritical: both sets of eigenvalues will converge to the same limit so long as $H_n$ converges in the manner
set out in the first part of Theorem \ref{weak}. But this has just been verified in the preceding subsections: Theorem \ref{thm:mainconv} establishing convergence
of the random potentials, or Assumption 1 (of Theorem \ref{weak}), and   Propositions  \ref{p:osc} and \ref{c:finalstretch} combined verifying
the growth/oscillation conditions, Assumptions 2.

To identify the Stochastic Airy Operator $\SAO$ only requires replacing the simple scalings on the $X_n(x)$ and $Y_n(x)$ sums in Theorem \ref{thm:mainconv} with what is actually
given via \eqref{eq:mn} and \eqref{summedpotentials}. With $a = a^{\dagger}(0)$ and $b = b^{\dagger}(0)$ for short, the convergence,
$$
   \left( \frac{n}{\tau b^2} \right)^{1/3}
  \sum_{k=  \lfloor c \log n \rfloor}^{ \lfloor x (b n /\tau)^{1/3} \rfloor } (a - A_k )  + 2 (b  - B_k)   \rightarrow \frac{1}{2} x^2 + \frac{2}{\sqrt{\beta}} (b \tau)^{1/2}  \cdot \frac{1}{\tau^{1/3} b^{2/3}}  W( (b/\tau)^{1/3} x),
$$
may be read off from \eqref{eq:firstconv}-\eqref{eq:stats}.  By Brownian scaling the right hand side is equivalent in law to $\frac{1}{2} x^2 + \frac{2}{\sqrt{\beta}} W_x$ as is required.

For convergence of eigenvectors, note that we do have from Theorem \ref{weak} that the eigenvectors of $H_n$ and $H_n+ z_n e_{11}$ converge to the same limit, the eigenvectors of $\SAO$.  Now denote by $\mathcal{T}_n$ the full matrix \eqref{e:scaled} and by $\underline{\mathcal T_n}$ its (positive definite) lower bound from Theorem \ref{t:beginning},  scaled to our setting. That is, $\underline{\mathcal T_n} =   y_n I_{|[1, c \log n]}  + [H_n + z_n e_{11}]$, in which the previous perturbed minor is padded by zeros in the first $c  \log n$ rows and columns.

We now show, by induction on $k$ that $|\langle {\underline \varphi}_{nk}, \varphi_{nk},\rangle|\to 1$  i.e. the normalized eigenvectors are close in norm (up to sign). 

Let $\underline A_{n,k}$ denote the $k$-dimensional subspace spanned by the lowest eigenvectors of  $\underline{\mathcal T_n}$.
Assume that $n$ is large enough so that $y_n>\underline \lambda_{n,k+1}\ge \underline \lambda_{n,k}$, and  the elements of $\underline A_{n,k}$
vanish on the first $c\log n$ coordinates.

Write $\varphi_{n,k}=\varphi'+\varphi''+\varphi'''$, where  
$\varphi''$ is the orthogonal projection to $\underline A_{n,k-1}$,  
$\varphi''$ is the orthogonal projection to the 1-dimensional space $\underline \varphi_{n,k}$. Then, granted that
$y_n\ge \underline{ \lambda}_{n,k+1}$, we have 
$$
\lambda_{n,k}=\langle \varphi, \mathcal T_n \varphi \rangle \ge \langle \varphi, \underline{\mathcal T}_n \varphi \rangle \ge \underline \lambda_{n,1}\|\varphi'\|^2 + \underline \lambda_{n,k}\|\varphi''\|^2+ \underline \lambda_{n,k+1}\|\varphi'''\|^2
$$
where the 
last inequality follows from the eigenvector decomposition of $\underline{\mathcal T}_n$.

Writing the inequality between the leftmost and rightmost expressions as in
$$
(\lambda_{n,k} - \underline \lambda_{n,1})\|\varphi'\|^2 +(\lambda_{n,k} - \underline \lambda_{n,k})\|\varphi''\|^2  \ge (\underline \lambda_{n,k+1}-\lambda_{n,k})\|\varphi'''\|^2,
$$
we claim that the the left hand side converges to zero. Indeed, by the inductive
hypothesis $\underline A_{n,k-1}$  is close to the subspace spanned by the first $k-1$ eigenvectors of $\mathcal T_n$, so $\|\varphi'\|
\to 0$, and the coefficient stays bounded.  For the second term, we have established that $\lambda_{n,k}-\underline \lambda_{n,k}\to 0$. 

Now  $\SAO$ has almost surely discrete eigenvalues, and so  $\underline \lambda_{n,k+1} - \lambda_{n,k}>C>0$ for a random constant $C$ and
all $n$ large enough. Thus for the right hand side to converge to zero it has to be the case that  $|\varphi'''|\to 0$. Hence  $|\varphi''|\to 1$, as claimed. 

Thus the lower eigenvectors of $\mathcal T_n$ and $\underline {\mathcal T_n}$ are close in norm. The latter are vanishingly small shifts of the eigenvectors of 
$H_n+ z_n e_{11}$, which are close to the eigenvectors of $\SAO$, as required.
\end{proof}

\section{The nonregular case}
\label{s:nonregular}

Even for the classical values of the parameter $\beta = 1, 2 $ and $4$ there exist external fields $V$ for which one does not see Tracy-Widom limits at the
edge of the spectrum.   The relevant condition is understood through the behavior of the corresponding equilibrium measure $\mu_V$, again defined through the
minimizing
$$
   I(\mu) = \int V(s)  d \mu(s) + \int \int \log \frac{1}{ |s-t|} d \mu(s) d \mu(t)
$$
in the space of probability measures. If $\frac{V(s)}{\log |s| } \rightarrow \infty$ as $s \rightarrow \pm \infty$ the minimizer ($\mu_V$) exists and is unique. If it is further assumed that $V$ is real analytic, then it is the case that  has a density of the form
$$
    \psi_V(s) =   \frac{d \mu_V(s)}{ds} = \sqrt{(Q_V(s))_+}
$$
for a real analytic function $Q_V(x)$. Generically $Q_V$ has simple zeros and $\psi_V$ vanishes like a square root at the edge of its support \cite{KuilMcL}. In this case $V$ is called {\em regular} and the rescaled maximal eigenvalue has the appropriate Tracy-Widom law (see \cite{DG}, again for $\beta = 1, 2, 4$).

In general though it is possible for $Q_V$ to have a zero of order $4k+1, k=0, 1, \dots$ (but not of order $4k+3$).  The resulting ``higher-order Tracy-Widom"  laws which can result from such non regular $V$'s have been studied in \cite{Claeys} at $\beta = 2$.  The eigenvalue fluctuations in these settings is of order $n^{ 2/(4k+3)}$.

Throughout the above we have made essential use of the assumption that $V$ is (uniformly) convex, specifically to control the regularity of the minimizers as well as the concentration of the field about the minimizers.  In fact, the only simple geometric condition for $V$ to be regular is convexity. Nonetheless, our framework provides conjectural
operator limits, of type Laplacian plus random potential,  describing the limiting spectral edge for any non regular $V$.

To explain, one first notices that the shape of the potential (linear plus a constant multiple of white noise) that characterizes the Stochastic Airy Operator is tied directly to the differentiability of the local minimizers $a^{\dagger}(x), b^{\dagger}(x)$ at zero. Really what is important is the shape of the edge $\mathcal{E}(x) =  a^{\dagger}(x) + 2 b^{\dagger}(x)$
tied to the family of potentials $x \mapsto \frac{1}{1-x} V$, again about zero. Recall Remark \ref{r:momentcond} at the end of Section \ref{s:outline}. It has already been noted in
\cite{KuilMcL} that the behavior of $\psi_V$ at $\mathcal{E} = \mathcal{E}(0)$ determines that of $\mathcal{E}(x)$.  Lemma 8.1 of that paper implies:
\begin{equation}
\label{eq:densitytoedge}
   \mbox{ If } \psi_V(t)  \sim  (\mathcal{E} -t)^{\frac{4k+1}{2}} \mbox{ as } t \uparrow  \mathcal{E} \mbox{ then }  \mathcal{E} - \mathcal{E}(\epsilon) \sim \epsilon^{\frac{1}{2k+1}}
   \mbox{ as } \epsilon \downarrow 0.
\end{equation}
Again, square-root vanishing of the density ($k=0$) produces linear/differentiable behavior of the edge $\mathcal{E}(x)  \sim \mathcal{E}(0) + \mathcal{E}'(0) x$.

After centering the tridiagonal operator by $\mathcal{E}(0)$, and so again factoring $b^{\dagger}(0) \times $ the discrete second-derivative operator, the summed potential
once more takes the form
$$
     y_n(x) =   m_n \sum_{k= q_n}^{ \lfloor m_n x \rfloor}  \left(  \mathcal{E}(0) - (A_k + 2 B_k)  \right),
$$
compare \eqref{summedpotentials}.
Here a variable lower limit $q_n$ has been inserted assuming there will again be a cutoff required. Before of course $q_n = c \log n$, what is important is that $q_n = o(m_n)$.
Now assuming that $A_k, B_k$ concentrate about appropriate local minimizes, at least in the vicinity of the edge, one has up to constants that
\begin{equation}
\label{heuristicmean}
   \EE{y_n(x)} \sim m_n \sum_{k= q_n}^{\lfloor m_n x \rfloor} (k/n)^{\frac{1}{2k+1}} \sim \frac{{m_n}^{\frac{4k+3}{2k+1}}}{n^{\frac{1}{2k+1}}} x^{\frac{2k+2}{2k+1}},
\end{equation}
granted \eqref{eq:densitytoedge}. This implies that $m_n$ must be chosen so that $m_n \sim n^{\frac{1}{4k+3}}$, which
corresponds to the general fluctuation exponent  known to hold
in the $\beta =2$ case (the fluctuation is $O(m_n^2)$), as it should.  Note that this can be cast in terms of the concentration of the random ``edge" $A_k + 2 B_k$, without discussion of separate $(a,b)$-minimizers.

For the variance, the formula in the regular case  (see Proposition \ref{thm:mainconv}) can be written in an intuitive manner as in
$$
 - 2 \beta \mathcal{E}(0) \mathcal{E}'(0)  x = - \beta (\mathcal{E}^2)'(0) x   \sim \beta  n^{2/3} \sum_{k \le n^{1/3} x}   - (\Delta \mathcal{E}^2)( {k}/{n} ),
$$
where $\Delta{f}(t)$ is short for ${f}(t +1/n) - {f}(t)$. On the left hand side, one recognizes the formula given previously, with constant  $-4 \beta b^{\dagger}(0) \mathcal{E}'(0)$, noting that it can be assumed that $a^{\dagger}(0) = 0$ by a simple shift in which case $\mathcal{E}(0) = 2 b^{\dagger}(0)$. The right hand side captures that the (Gaussian) fluctuation a single step of
the walk $k \mapsto A_k + 2 B_k$ should scale with the difference-square of its mean profile, $\mathcal{E}(k/n)$.  This prompts
\begin{equation}
   \label{heuristicvar}
   \Var [ y_n(x) ]  \sim   m_n^{2}  \left(\mathcal{E}^2(0) -  \mathcal{E}^2( x m_n/n) \right)  \sim  m_n^2 ( x m_n/n)^{\frac{1}{2k+1}},
\end{equation}
as the proposed variance, neglecting constants.  Note that the exponents of $m_n$ and $n$ match those in \eqref{heuristicmean}.

Restoring the appropriate constant factors, \eqref{heuristicmean} and \eqref{heuristicvar} allow us to formulate the following.

\begin{conjecture}\label{c:sk} Assume $V$ is nonregular: with $k \ge 1$, let  $\psi_V(t)  \sim  (\mathcal{E} -t)^{\frac{4k+1}{2}}$ as $t \uparrow  \mathcal{E} = \mathcal{E}(0)$, the rightmost edge of
the support of $\psi_V$. In line with \eqref{eq:densitytoedge} define the constant $c$ by
$$
      \lim_{\epsilon \downarrow 0} \frac{ \mathcal{E} - \mathcal{E}(\epsilon)}{  \epsilon^{\frac{1}{2k+1}}} = c,
$$
and set $\gamma = c^{-2/3} (\mathcal{E}/2)^{-1/3}$. Then we have that
$$
     H_{n,k} =  \gamma n^{{2}/{4k+3}} ( \mathcal{E} I -  T_n)
$$
converges in the sense of Theorem \ref{weak} (i) to the operator
$$
   \mathcal{S}_{\beta, k} =  - \frac{d^2}{dx^2} + x^{\frac{1}{2k+1}} +  \frac{2}{\sqrt{\beta}} x^{- \frac{k}{2k+1}} W'(x),
$$
on the half-line with Dirichlet conditions at the origin. In particular, the ordered eigenvalues/eigenvectors
of $H_{n,k}$ converges jointly in law to those of $\mathcal{S}_{\beta, k} $.
\end{conjecture}

\end{document}